%% file: preprint.tex
\documentclass[10pt]{article}
\usepackage[a4paper, total={6.5in, 10in}]{geometry}
\input{preamble}
\usepackage{lscape}
\usepackage{authblk}
\usepackage{tabularray}
\usepackage[sorting = none, maxcitenames=1]{biblatex}
\title{Mean-field imitation dynamics on fast assortative networks}
\author[1]{Benedict Russell}
\author[2]{Andrew Nugent}
\author[3, 4]{Jacques Bara}
\affil[1]{Mathematics Institute, University of Warwick}
\affil[2]{Department of Mathematics, University College London}
\affil[3]{Center for Development Research, University of Bonn}
\affil[4]{Transdisciplinary Research Area Sustainable Futures, University of Bonn}

\begin{document}

\date{}
\maketitle

\begin{abstract}
The emergence of cooperation in structured populations is fundamental to the success of human societies. Physical and online networks can drive behavioural change by altering who people interact with, thereby modifying social pressures. In this paper, we study imitation dynamics in a population of self-interested agents playing a continuous strategy Prisoner's Dilemma on a dynamically evolving weighted network. In the fast-network regime, we incorporate the edge weights into the strategy evolution before deriving and analysing the large population mean-field limit. Without noise, we establish well-posedness and show the solution collapses to a single Dirac mass. For initially separated clusters, we identify a payoff threshold and sufficient conditions for the overall level of cooperation to increase. We then introduce stochastic strategy updates, and obtain a non-local Fokker-Planck equation in the mean-field limit. We rigorously prove existence and uniqueness of stationary distributions, and show linear stability under sufficient noise. Numerics illustrate that noise can transform the deterministic consensus into stable cooperative stationary behaviour. These findings show that the fast adaptive interactions and stochastic exploration can jointly support the emergence of stable cooperation at a population level.
\end{abstract}

\section{Introduction}
Understanding the emergence of cooperation and its stability is a central problem in evolutionary game theory \cite{nowak_five_2006}. At the heart of this problem is the tension between individual and collective welfare: individuals can obtain an immediate personal gain by acting selfishly, whilst mutual cooperation would obtain better long-term outcomes for the population. This is captured by social dilemmas such as the Prisoner's Dilemma, where defection is individually optimal but widespread adoption causes the breakdown of societal welfare. These choices are concretely present in examples such as public-goods provision \cite{conybeare_public_1984}, tax compliance \cite{gottlieb_tax_1985} and vaccine uptake \cite{kabir_how_2021}.

Previous works have shown that the outcomes of such games are dependent on the interaction structure of the population \cite{nowak_five_2006, rand_dynamic_2011, santos_cooperation_2006, bara_enabling_2022}. Instead of interacting in a well-mixed population, agents typically interact on some underlying network which encodes social, technological or biological ties. Such static structures have been shown to support cooperation theoretically \cite{fotouhi_evolution_2019, ohtsuki_simple_2006} and in simulations \cite{santos_scale-free_2005}. However, in reality, these networks are rarely static. Indeed, social connections are continuously being formed, strengthened, weakened, or broken in response to the behaviour of agents. By breaking ties with defectors, pro-social agents can find other cooperators \cite{melamed_prosocial_2017} and create stable clusters \cite{fehl_co-evolution_2011}. 

Allowing agents to adjust their ties has been shown to substantially change the emergence and sustainability of cooperation \cite{zhang_opting_2016, anastassacos_partner_2020, izquierdo_successful_2026}. In particular, simulations in which no partner selection is enforced, but instead is learnt with reinforcement learning, have been shown to lead to cooperation \cite{leung_learning_2024, fan_colearning_2025, defection_russell2026}. Theoretical and computational work has shown that partner update rules can bring about cooperation: \cite{graser_repeated_2025} show changing partners can be more effective than retaliation by defection, whilst \cite{pacheco_active_2006} find that fast active linking, corresponding to altering the payoff matrix in a well-mixed population, led to increased cooperation with heterogenous edge dynamics \cite{segbroeck_coevolution_2010}. \cite{zheng_simple_2017} find a stable mix of cooperators and defectors emerges in a mean-field opting out model, where pairs of agents are kept only if they cooperate. This has been extended to reinforcement learning populations, highlighting the role of population variance in supporting cooperative outcomes \cite{russell2026dynamicspolicygradient}. In each of these, the precise effects are dependent upon strategy update, tie breaking rules, and the relative time-scales of the co-evolutionary process. For example, it was recently shown that when a connection can quickly respond to the action of the opponent, cooperators can find and interact with other cooperators, leading to assortative patterns such as the core-periphery structure \cite{bara_enabling_2022}.

We study a co-evolutionary Prisoner's Dilemma on a weighted graph, where the weight captures the interaction strength. Each agent adopts a continuous strategy $x_i \in [0,1]$ representing its level of cooperation. Continuous strategies capture the agent's probability of cooperating, and are used as opposed to discrete binary strategies which have limited flexibility. The strategies evolve through pairwise comparison, so an agent is more likely to move its strategy towards a neighbour with higher payoff. Concurrently, the interaction network evolves endogenously: links between agents are increased when the agents are more cooperative, thus inducing an assortative network with respect to the cooperation level.

Research on the evolution of networks is mainly split into two types of update: edge rewiring \cite{zheng_simple_2017, defection_russell2026, izquierdo_leave_2014} and continuous weight dynamics \cite{bara_enabling_2022, nugent_evolving_2023, nugent2026emergent,gkogkas2021continuum,ayi2021mean}. The former is often used to model systems where connections are formed and severed at discrete time steps, whereas the latter models the gradual formation and strengthening or weakening of connections. Continuous weight dynamics can represent the frequency of interaction under edge rewiring; sampling a new partner proportional to the edge weight normalised by the total degree provides the scaling to go from the agent-based model to a continuous ODE model \cite{nugent2024bridging}. In this paper, we adopt these continuous weights, with the interpretation that agents play many games with their neighbours and are paired accordingly. We do not model the individual games played by the agents, but instead consider the expected payoffs obtained in the limit of playing a large number of games with their neighbours. 

We are particularly interested in the macroscopic behaviour of this coupled system in the fast network regime, where the interaction network adapts quickly relative to the strategy updates. This builds on similar timescale analysis in active linking models \cite{pacheco_active_2006, segbroeck_coevolution_2010} and evolving networks \cite{bara_enabling_2022, nugent_evolving_2023, guo_network_2023}. At the microscopic level, finite agent ODE or SDE systems can display rich dynamics, but their trajectories are sensitive to the exact initial strategies and stochastic noise. This makes providing a robust understanding of the population-level outcomes difficult. Instead, we adopt a mean-field perspective and study the large population limit through the population density $\mu(t,x)$, representing the distribution of cooperation levels $x$ at time $t$. This approach is based in interacting particle systems: in the large population limit, the empirical distribution converges to the law of a McKean-Vlasov process \cite{propogation_review}. Similar mean-field approaches have been used to model flocking \cite{HASKOVEC201342}, population of reinforcement learning agents, \cite{hu_modelling_2019, leung_modelling_2022, russell2026dynamicspolicygradient}, as well as opinion dynamics \cite{nugent_opinion_2025}. In the regime where the interactions evolve much faster than the strategies, the weight dynamics reduce to a function of the current strategy which gives the exchangeability of agents with the same cooperation level, which is necessary to take the mean-field limit. 

This work is positioned at the intersection of evolutionary game theory, dynamic social networks, and mean-field models of interacting particle systems. The goal is to connect the assortative mechanisms to a macroscopic PDE, and characterise how fast network adaptation shapes the evolution of cooperation in large populations. Note that this approach differs from mean-field games, where agents solve individual optimal control problem whilst interacting with the mean-field \cite{carmona_meanfieldgames}. In Section \ref{Section: model}, we introduce the finite agent model and formally derive the mean-field PDE in the fast-network regime. Section \ref{Section: No noise} studies the mean-filed PDE under deterministic strategy updates, proving well-posedness, converges to consensus, and explicit two-cluster results revealing a payoff threshold. Section \ref{Section: With noise} adds stochastic exploration to the strategy updates, specifically showing the existence, uniqueness and linear stability of stationary equilibria for the corresponding mean-field PDE. Section \ref{Section: Conclusion} discusses implications, limitations and future research directions.

\section{Model}\label{Section: model}
Consider a finite population of $N$ agents on a network, with the edge weight between agents $i$ and $j$ is denoted by $a_{ij}(t)$.  Without edge normalisation, an agent could update more frequently by having more neighbours; normalising by the total degree, $k_i = \sum_{j\in N}a_{ij}$, ensures the strategy dynamics are comparable across agents and converts edge strength to a relative interaction frequency \cite{nugent2024bridging}. Each agent has a strategy $x \in \Omega$; in this paper we adopt $\Omega =[0,1]$ as the strategy corresponds with the probability of cooperation, where $x=0$ indicates pure defection and $x=1$ pure cooperation. 

Each iteration represents a Prisoner's Dilemma \cite{nowak_five_2006}; cooperators pay a cost of $c$ to provide a benefit of $b$ to each of their neighbours. Defectors pay no such fee, and no benefit is received by their neighbours. This induces a payoff for each agent, given by the rewards gained minus the contribution to each neighbour:
\begin{align*}
    \pi_i = b\bigg(\sum_{j=1}^Na_{ij} x_j \bigg) - c k_ix_i.
\end{align*}

For strategy imitation, we consider the pairwise comparison rule, whereby agent $i$ compares its payoffs with a neighbour $j$, and is more likely to move towards $x_j$ if $j$ has a higher payoff \cite{pinheiro_linking_2016}. This is given by the Fermi update probability
\begin{align*}
    p_{ij} &= \frac{1}{1+\exp[-\beta(\pi_j-\pi_i)]}\\
    &= \frac{1}{1+\exp[-\beta \sum_{l=1}^Nbx_l(a_{jl}-a_{il}) -c(x_ja_{jl}-x_ia_{il})]},
\end{align*}
where $\beta\geq0$ is a selection-intensity parameter. Small $\beta$ corresponds to weak selection, whilst large $\beta$ yields strong selection. In the continuous strategy setting, this acts as an interaction function, encouraging movement towards strategies with higher payoffs. The individuals' strategies evolve according to the coupled system of ODEs,
\begin{align}\label{eqn: strategy ode standard speed}
    \frac{dx_i}{dt} &= \frac{1}{k_i}\sum_{j=1}^N p_{ij}a_{ij}(x_j-x_i)
\end{align}
so that each agent moves towards an average of their neighbour's strategies, with contributions weighted by the relative interaction frequency with each neighbour. The ODE has a similar form to other alignment processes such as those in \cite{motsch2014heterophilious, nugent2024bridging, brooks2024emergence, couzin2002collective}.

Concurrently, the network evolves on a separate timescale $\tau$. To model the assortative network, we allow the edge weights to evolve with the agents' strategies. In particular, we use the extreme popularity model introduced in \cite{bara_enabling_2022} in which cooperators are always befriended and defectors unfriended. The evolution of the weights under this mechanism is given by
\begin{align}\label{eqn: adjacency ode}
  \frac{da_{ij}}{dt} &= \frac{1}{\tau}\bigg(\frac{x_i+x_j}{2} - a_{ij}\bigg).
\end{align}
Under these dynamics, edges between cooperative agents are strengthened, whilst edges with defecting agents are weakened. As such, the evolution is endogenously assortative with respect to the strategies. Other possible weight dynamics include active linking \cite{pacheco_active_2006}, exogenously imposed \cite{li_evolution_2020}, and friend-of-a-friend \cite{nugent_evolving_2023}. The mechanism given by \ref{eqn: adjacency ode} is chosen as it quickly induces an assortative network structure.

\subsection{Fast network dynamics}

Motivated by how responsive networks can induce cooperative behaviour \cite{bara_enabling_2022, guo_network_2023}, we consider the fast-slow system, where network dynamics happen much faster than the strategy update. This enables the weight between two agents to be explicitly written as a function of their two strategies. Consequently, the population is exchangeable: agents with the same strategy will have the same connections. In turn, this enables the derivation of the mean-field limit. We can achieve the fast-network limit by sending $\tau \rightarrow 0$; since the strategies are continuous and $a$ is linear, by Theorem 15.2 from \cite{pavliotis_stuart_2008} the system \eqref{eqn: strategy ode standard speed}, \eqref{eqn: adjacency ode} converges to
\begin{subequations}
\begin{align}
    \frac{dx_i}{dt} &= \frac{1}{k_i}\sum_{j=1}^N p_{ij}a_{ij}(x_j-x_i),\\
    a_{ij} &= \frac{x_i+x_j}{2}.
\end{align}
\end{subequations}
We note that $a_{ij}$ is not a function of the agent's indices $i,j$, but of the strategies $x_i$ and $x_j$. Scaling the denominator into the timescale separation, this functional definition can be substituted into the strategy evolution, reducing it to the singular autonomous ODE
\begin{align}\label{eqn: agent update ODE}
    \frac{dx_i}{dt} &= \frac{1}{k_i}\sum_{j=1}^N p_{ij}(x_j^2-x_i^2).
\end{align}
Figure \ref{fig: finite population ODE} shows the evolution of a finite number of agents where the initial policies are uniformly distributed. As the agents converge towards consensus (single strategy), the adjacency converges to the uniform interaction function (Equation \eqref{eqn: adjacency ode} converges to the consensus population strategy $x$). In comparison, the complete network with uniform weights converges towards consensus at a lower level of cooperation.

\begin{figure}
     \centering
\includegraphics[width=0.8\textwidth]{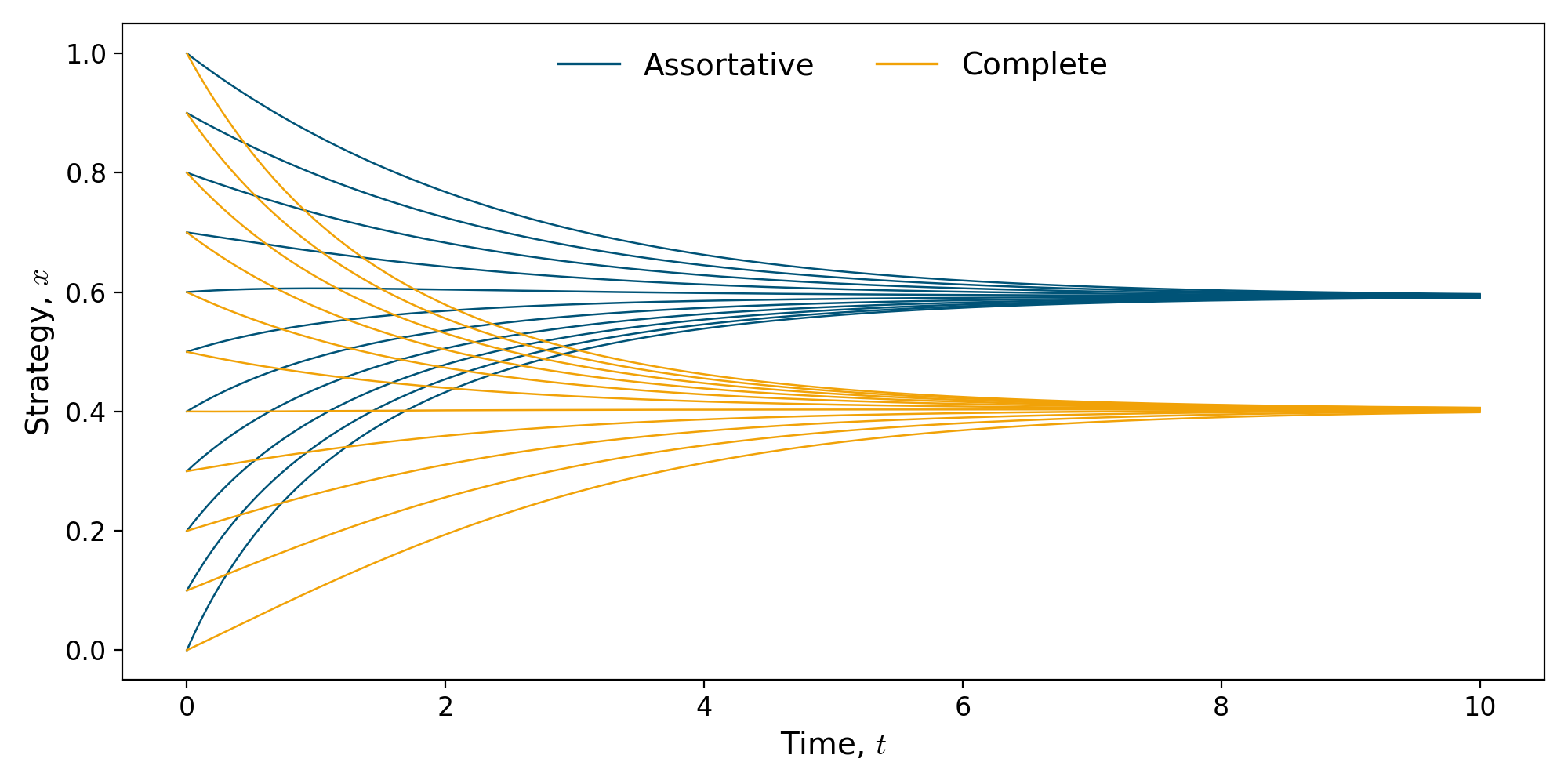}
     \caption{A comparison of fast network assortative dynamics against a complete graph in a finite population. In both settings, the population reach a single strategy. The assortative mechanism induces a higher level of cooperation compared to a fixed, complete network. Parameters are $b=5,c=1,\beta =1,N=11$ and the weights are initially uniform.}
     \label{fig: finite population ODE}
\end{figure}

Whilst simple, ODE solutions are sensitive to the initial condition sampled from the initial strategy distribution, and therefore it is not clear how to interpret convergence for a given distribution. To solve this, we aim to analyse the macroscopic dynamics in this fast-slow system, using a mean-field PDE to understand the distributional behaviour in a large population limit.

\subsection{Formal derivation of the mean-field limit}
In this section, we derive a mean-field limit of the form $\mu(t,x)$ which describes the density of agents with strategy $x$ at time $t$.  We begin by deriving a PDE satisfied by the empirical density \cite{jabin_meanfieldlimit_2017} in a system without noise, which is defined by
\begin{align}\label{eqn: empirical measure}
    \mu(t,x) = \frac{1}{N}\sum_{i=1}^N \delta_{x_i}(x)
\end{align}
where $\delta_x$ denotes a Dirac mass centred at $x$. Let $\varphi(x)$ be an infinitely differentiable, compactly supported test function which vanishes on $\partial\Omega$, and denote the Fermi probability as $p_{\underline{x}}(x_i,x_j) := p_{ij}$, with $\underline{x}$ highlighting the dependence upon the strategy profile of the whole population. Then
\begin{align*}
    &\int_\Omega \partial_t\mu(t,x)\varphi(x)dx\\ &= \frac{d}{dt}\int_\Omega\frac{1}{N}\sum_{i=1}^N \delta_{x_i}(x) \varphi(x)dx\\
    &=\frac{1}{N} \sum_{i=1}^N \varphi'(x_i) \Bigg[\frac{1}{ k_i}\sum_{j=1}^N a_{ij}p_{\underline{x}}(x_i,x_j)(x_j-x_i)\Bigg]\\
    &=\int_\Omega\frac{1}{N} \sum_{i=1}^N \varphi'(x) \Bigg[\frac{1}{ \sum_{k=1}(x+x_k)}\sum_{j=1}^N (x_j^2-x^2)p_{\underline{x}}(x,x_j) \Bigg] \,\delta_{x_i}(x) \, dx\\
    &=\int_\Omega\mu(t,x) \,\varphi'(x) \Bigg[\frac{1}{\sum_{k=1}(x+x_k)}\sum_{j=1}^N (x_j^2-x^2)p_{\underline{x}}(x,x_j)\Bigg] dx\\
    &=\int_\Omega\mu(t,x) \,\varphi'(x) \Bigg[ \frac{1}{ (x+\int_\Omega z\mu(t,z)dz)} \int(y^2-x^2)\mu(t,y)p_{\underline{x}}(x,y)dy\Bigg] dx\\
    &=-\int_\Omega \,\varphi(x) \, \partial_x \Bigg( \mu(t,x)\Bigg[ \frac{1}{(x+\int_\Omega z\mu(t,z)dz)} \int(y^2-x^2)\mu(t,y)p_{\underline{x}}(x,y)dy\Bigg] \Bigg) dx.
\end{align*}
Where the final lines follows from integration by parts. In particular, we assume $\varphi$ vanishes on $\partial\Omega$, since $\Omega$ is itself compact. Therefore, $\mu$ is a weak solution to
\begin{align}\label{eqn: mean field pde without noise}
    \partial_t\mu(t,x) +\partial_x\Bigg( \frac{\mu(t,x)}{x+\int_{\Omega} z\,\mu(t,z)\,dz} \int_{\Omega} (y^2-x^2)\,\mu(t,y)\,p_\mu(x,y)\,dy\Bigg) = 0 \,.
\end{align}
We now address the selection probability, $p_{\underline{x}}(x,y)$. In the ODE model, the payoff difference functional simplifies to
\begin{align*}
    \pi_j - \pi_i &= \frac{1}{2}(x_j-x_i)(b-c)\sum_{l=1}^Nx_l - \frac{cN}{2}(x_j-x_i)^2, 
\end{align*}
and so by rescaling the mean-field interaction scaling from $\beta$ to $\beta \times 2/N$, this converts the absolute payoff difference to an average payoff difference. The mean-field Fermi probability is 
\begin{align}\label{eqn: selection function p}
    p_\mu(x,y) &= \frac{1}{1+\exp\big(-\beta(y-x)[\int_\Omega (b-c)z\,\mu(t,z)\,dz - c(y+x)]\big)},
\end{align}
where we note the dependency changes from the strategy profile $\underline{x}$ to the distribution $\mu$. Note that substituting in the empirical measure \eqref{eqn: empirical measure} into \eqref{eqn: selection function p} and rescaling $\beta$ exactly recovers the finite population Fermi function. In the large population limit, the empirical distribution can converge to a smooth density. Consequently, it is important to analyse the evolution of both smooth and measure-valued initial conditions. We will analyse the solutions induced by these dynamics in Section \ref{Section: No noise}. The addition of noise at the microscopic level creates a non-trivial extension, and is discussed in Section \ref{Section: With noise}. Without noise, the natural dynamics preclude the need for boundary conditions (Theorem \ref{theorem: no noise existence and uniqueness}); where noise is present, we impose no flux boundary conditions on the boundary of $\Omega$. In both instances we specify that the initial condition is given by $\mu_0(x)$ for $x \in \Omega$.

\section{Deterministic microscopic dynamics} \label{Section: No noise}
In this section, we first study the well-posedness of the PDE given by \eqref{eqn: mean field pde without noise} which arises from deterministic microscopic dynamics \eqref{eqn: agent update ODE}. By considering two clusters of agents, one initially made up of cooperators and the other of defectors, we then provide exact results for the dynamics and long term behaviour of the model for different parameter regimes, providing a concrete understanding of why cooperation emerges under specific regimes. For readability, proofs in this section have been moved to the Supplementary Material.

\subsection{Well-posedness}
We begin by showing existence of a solution when the initial probability distribution, $\mu_0(x)$, is suitably smooth. To do this, we first show that the mean of the distribution is bounded below for all finite times, and use this to prove regularity conditions on the flow of the PDE. Equation \eqref{eqn: mean field pde without noise} can be neatly expressed by defining the mean ($m_\mu$), a velocity ($V_\mu$), and an interaction function ($I_\mu$). In particular, 
    \begin{align}
    \partial_t\mu(t,x) +\partial_x (\mu(t,x) V_\mu(t,x)) &= 0
    \end{align}
    where 
    \begin{align*}
        m_\mu(t) = \int_\Omega z \mu(t,z) dz,\quad 
        V_\mu(t,x) = \frac{I_\mu(t,x)}{x+ m_\mu(t)}, \quad  I_\mu(t,x) = \int_{\Omega} (y^2-x^2)\,\mu\,p_\mu(x,y)\,dy.
    \end{align*}
This puts the PDE into the standard form of a continuity equation, in which the velocity determines the flow. The only concern regarding existence of a solution lies in the denominator of $V_\mu$: there is a possibility that the PDE pushes all mass towards zero which would break the continuity of $V_\mu$ at the left boundary $x=0$ as there would be a division by zero if $m_\mu$ is zero. As such, the next Lemma provides a useful lower bound on the mean, such that for any finite time interval we avoid this scenario. 
\begin{lemma}\label{lem: mean bounded below}
    On any finite interval $[0,T]$, the mean is bounded below. Specifically,
    \begin{align*}
        m_\mu(t) \geq m_{\mu_0}e^{-{t}}.
    \end{align*}
\end{lemma}
This bound will propagate throughout the Lipschitz bounds below, and allow us to show existence and uniqueness on any finite time interval, $[0,T]$. For simplicity, denote $\delta_T=m_{\mu_0}e^{-T}$ as this lower bound at the final time step $T$. Then, for all $t\leq T$, we have that 
$m_\mu(t) \geq \delta_T$. We note that in the trivial case of $m_{\mu_0}=0$ all the mass must necessarily be at zero: the population consists of pure defectors. The individual strategy evolution in \eqref{eqn: agent update ODE} is zero for all agents -- as there are no cooperators to imitate -- so the population is at a steady state. The next two Lemmas provide the Lipschitz guarantees on the advection in both the strategy space and the distribution when the initial mean is non-zero.
\begin{lemma}\label{lem: p is Lipschitz W_1}
    The function $p_{\mu}(x,y)$ \eqref{eqn: selection function p} is Lipschitz continuous in $\mu, x$ and $y$, with constant $L_p$. 
\end{lemma}

\begin{lemma}\label{lem: V is Lipschitz W_1}
    On any finite time interval $[0,T]$, the function $V_\mu(t,x)$ is Lipschitz continuous in $x$ and $\mu$, with constant $L_V$. This constant is dependant upon the time $T$ through the lower bound on the mean, $\delta_T$.
\end{lemma}

These conditions prove sufficient to show the existence of a unique solution to Equation \eqref{eqn: mean field pde without noise}. In particular, we look for a solution in the space of compactly supported probability measures on the real line,  $\mathcal{P}_c(\mathbb{R})$, before showing the solution holds for the domain $\Omega$.
\begin{theorem}\label{theorem: no noise existence and uniqueness}
    Let $\mu_0 \in \mathcal{P}_c(\Omega)$, then on any finite time interval $[0,T]$, there exists a unique solution $\mu(t,x) \in C^0([0,T], \mathcal{P}_c(\Omega))$ to Equation \eqref{eqn: mean field pde without noise}. Moreover, the solution is given by
    \begin{align*}
        \mu(t) = X(t,\cdot)_\#\mu_0,\quad \dot{X}(t,\cdot) = V_\mu(t,X(t,\cdot)),\quad X(0,x)=x.
    \end{align*}
    where $X(t,\cdot)_\#\mu_0$ is the push forward of $\mu_0$ by the characteristic flow.
\end{theorem}

\subsection{Long term behaviour}
The solution in $\mathcal{P}_c(\Omega)$ preserves absolute continuity \cite{Bonnet2017ThePM}, ensuring that any initial density remains a density. In the next result, we provide an exact bound on how fast the maximum point of the distribution can grow.
\begin{proposition}\label{prop: no noise solution maximum bound}
    Let $\mu_0 \in \mathcal{P}_c(\Omega)$ be absolutely continuous with density $\mu_0 \in L^\infty(\Omega)$ and mean $m_{\mu_0} >0$. Then on any finite time interval $[0,T]$, the solution remains absolutely continuous, and
    \begin{align*}
        \|\mu(t,\cdot)\|_\infty \leq \|\mu_0\|_\infty \exp\bigg[t\Big(\frac{1}{\delta_T^2} + \frac{2+c}{\delta_T} + \frac{\beta b}{4}\Big)\bigg]. 
    \end{align*}
\end{proposition}
Finally, we show that the solution converges to a single Dirac measure. 
\begin{theorem}\label{thm: convergence of measure to a dirac}
    Let $\mu_0 \in \mathcal{P}_c(\Omega)$, then there exists $x^* \in [0,1]$ such that $\mu(t,\cdot) \rightharpoonup \delta_{x^*}$ as $t \rightarrow \infty$. That is, the measure converges weakly in $\mathcal{P}_c (\Omega)$ to a Dirac measure.
\end{theorem}
The result shows how the population adopts a single strategy, this is driven by the Fermi function which brings strategies closer together in a system without noise or exploration. In Section \ref{sec: two type dynamics}, we investigate how this limit point changes with the initial conditions and payoff structure.

To show how the fast assortative dynamics promote the emergence of cooperation, we compare against the dynamics of a well-mixed population. The corresponding mean-field limit is
\begin{align}\label{eqn: well-mixed mean field pde no noise}
    \partial_t\mu(t,x) +\partial_x\bigg( \mu(t,x)\int_{\Omega} (y-x)\,\mu(t,y)\,p^{\text{comp}}_\mu(x,y)\,dy\bigg) = 0 \,,
\end{align}
where $p^{\text{comp}}_\mu$ is the Fermi function for the complete graph (see the Supplementary Material for derivation).

\begin{proposition}\label{prop: unique solution for complete network}
    There exists a unique solution to Equation \eqref{eqn: well-mixed mean field pde no noise}. Moreover, as $t \rightarrow \infty$ the density converges weakly to a Dirac measure.
\end{proposition}

Since in both the fast assortative network and the complete network the strategies converge to a single strategy, the natural question arises of where these strategies lie. We expect that the former induces a higher level of cooperation as shown by the simulations (Figure \ref{fig: finite population ODE}). Figure \ref{fig: dirac comparison at T=10 (no noise)} shows the cooperation difference after evolution for $T=10$ seconds, with the same initial uniform distribution on the strategy and network. The payoff $b=3c$ represents a transition from a negligible difference (the fast network is similar to complete network) to a large disparity (where the fast network reaches much higher level of cooperation). We highlight that the payoff parameter $b$ has no impact in the complete case, as the payoff received from the neighbours is the same each agent and therefore cancels out when payoffs are compared. The difference is amplified as $\beta$ increases; for low values there is weak selection pressure causing a smaller difference between the final solution as the payoff difference has little effect on the dynamics. 

\begin{figure}[h]
     \centering
    \includegraphics[width=0.5\textwidth]{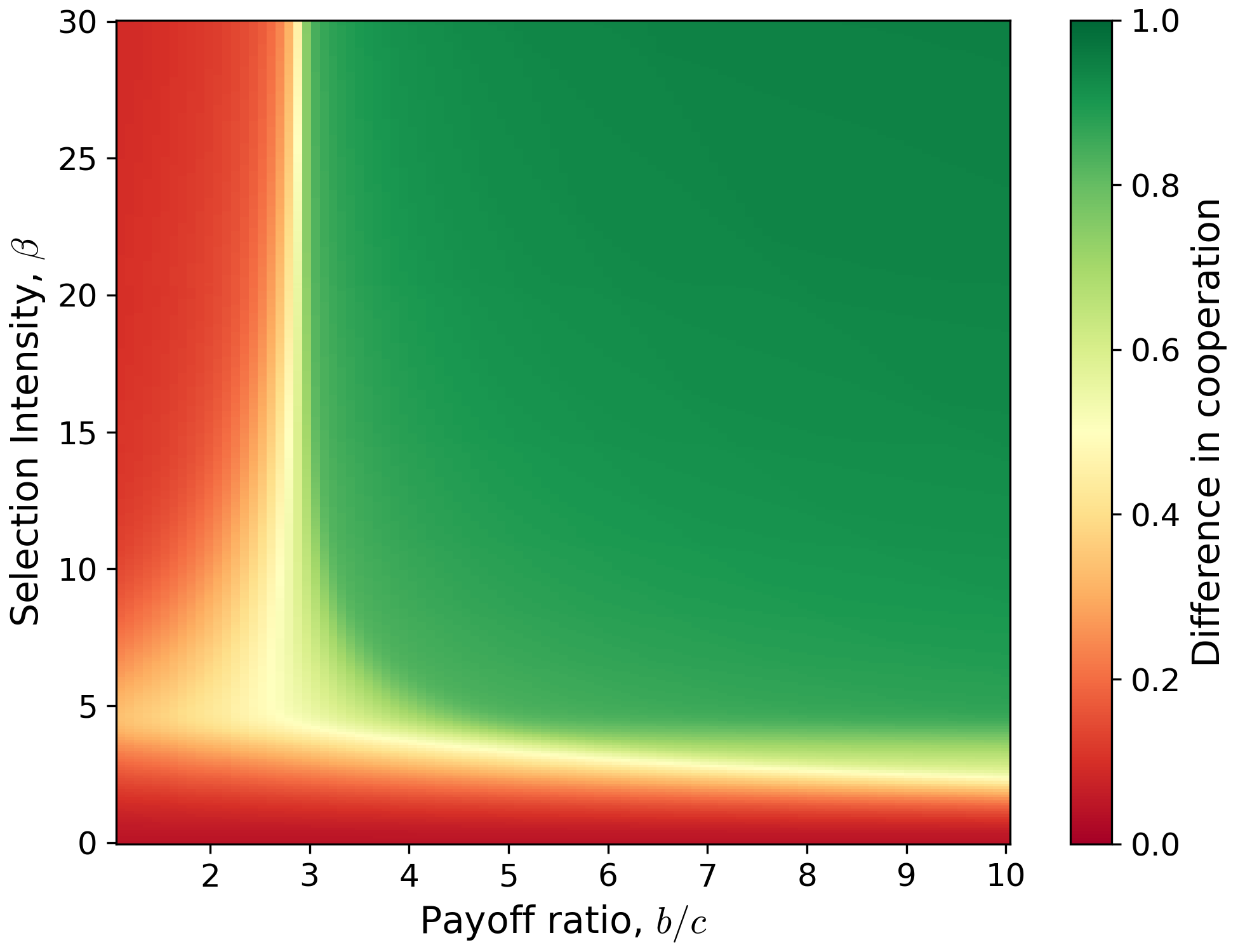}
     \caption{Difference between mean cooperation for fast assortative dynamics and a complete static network. The difference is always non-negative, and the dynamics transition near the payoff $b=3c$. Finite volumes method is used for $T=10$, with the mean at the final time represented by colour. The payoff $c=1$ is fixed whilst $b$ varies; the initial strategy distribution is uniform.}
     \label{fig: dirac comparison at T=10 (no noise)}
\end{figure}

To analyse why this behaviour emerges, we consider a simple but natural initial condition of a cluster of cooperators and a cluster of defectors. The interplay between these two groups can provide a foundational understanding for when and why cooperation emerges under the fast dynamics, but fails to materialise in a complete network.
\subsection{Two-type dynamics}\label{sec: two type dynamics}
We consider a population consisting of agents who either initially cooperate ($x=1$) or initially defect ($x=0$), thus reducing the analysis to the dynamics of these two initial strategy types, with the initial probability distribution being a sum of two Dirac measures centred at 0 and 1 respectively for defectors and cooperators. Denote $z(t)$ as the position of the initially defecting group and $y(t)$ the position of the initially cooperating group. Weighting the Diracs such that $w_z + w_y = 1$, the solution will track the position of these two clusters as they interact. The solution can be written as 
\begin{align*}
    \mu(t) = w_z \delta_{z(t)} + w_y\delta_{y(t)},\quad\text{where  } z(0) = 0,\; y(0) = 1.
\end{align*}
The population level of cooperation can be concretely measured with the mean. In this two-type setting, this simplifies to
\begin{align*}
    m_\mu(t) = w_zz(t) + w_yy(t).
\end{align*}
The atomic masses are transported by the velocity $V_\mu(t,x)$, with the weights remaining constant. Integrating over the Dirac measure, the evolution reduces to
\begin{subequations}\label{eqn: two type ode system}
\begin{align}
    \frac{dz}{dt} &=\frac{w_y(y^2-z^2)p_\mu(z,y)}{z+ m_\mu(t)} \,, \label{eqn: two type ode dz/dt}\\
     \frac{dy}{dt} &=\frac{w_z(z^2-y^2)p_\mu(y,z) }{y+ m_\mu(t)} \,, \label{eqn: two type ode dy/dt}
\end{align}
\end{subequations}
with selection term simplified to
\begin{align*}
    p_\mu(z,y) = \frac{1}{1 + \exp\big[-\beta\big(y-z\big)\Big(\big(b-c)m_{\mu}(t)-c(y+z)\Big)\big]} \,.
\end{align*}
\begin{corollary}\label{corollary: ode limit point}
    The ODE system given by \eqref{eqn: two type ode system} converges to a point such that $(z^*,y^*)=(L,L)$ for some $L \in [0,1]$.
\end{corollary}
\begin{proof}
    This follows immediately from the convergence of the PDE to a single Dirac in Theorem \ref{thm: convergence of measure to a dirac}.
\end{proof}
We aim to understand how the various parameters, specifically the weights and payoff structure, affect the limit point, $L$, reached by the evolving population. Define
\begin{align*}
    S(t):= (b-c)m_\mu(t) - c(y+z),
\end{align*}
then if $S = 0$ the imitation probability $p_\mu(z,y) = \frac{1}{2}$. $S$ captures the payoff difference between the cooperators and defectors clusters. Consequently, we require $S>0$ for the level of cooperation to increase. The following assumptions on the reward structure and population weights enable the selection mechanism to consistently favour one type over the other. 

\begin{assumption}\label{reward and weight assumption}
    The rewards are such that $b\geq 3c$ and the weights are initialised in the range $
    \frac{c}{b-c} \leq w_y \leq 1- \frac{c}{b-c}$, with $w_z = 1-w_y$. Recall that $w_y$ is the weight of the initially cooperating group.
\end{assumption}
This enforces a sufficient mass at the cooperative initial conditions, which is required to incentivise others to cooperate. The next result shows that under the above assumption the strategy space which satisfies $S(t)\geq0$ is positively invariant.
\begin{lemma}\label{lem: S_mu forward invariant}
Under Assumption \ref{reward and weight assumption}, then the region defined by
\begin{align*}
    X = \{(z,y) : 0 \leq z \leq y \leq 1, S &\geq 0 \}
\end{align*}
is non-trivial and forward invariant.
\end{lemma}
This invariant region enables us to provide a sufficient condition on the emergence of cooperative behaviour. The combination of a sufficiently high reward for cooperation and mass in the cooperating cluster enables the population level of cooperative behaviour to increase.
\begin{proposition}\label{proposition: cooperation is non-decreasing}
   Under Assumption \ref{reward and weight assumption}, then the total level of cooperation is non-decreasing.
\end{proposition}
The exact point to which the dynamics converge to is subject to a large number of parameters, so we provide an exact result for no selection pressure ($\beta=0$) before proving asymptotic behaviour as $\beta$ is increased.
\begin{proposition}\label{prop: limit point as function of beta}
    Let $L(\beta)$ denote the limit point, defined in Corollary \ref{corollary: ode limit point}, as a function of the selection intensity. The following holds:
    \begin{enumerate}
        \item $L(0) = \sqrt{\frac{w_y(1+w_y)}{2}}$
        \item If $\mu_0$ satisfies Assumption \ref{reward and weight assumption} with strict bounds on $(w_z,w_y)$, then\\ $\lim_{\beta\rightarrow\infty}L(\beta) = 1$ and $L(\beta)$ is monotone non-decreasing in $\beta \in [0,\infty)$.
    \end{enumerate}
\end{proposition}
\begin{proof}
    We provide a proof sketch here, with a full derivation in Appendix \ref{Appendix: No noise}. Let $\beta =0$, then the selection term $p = \frac{1}{2}$. The coupled ODEs can then be put in a non-parametric from  which yields
    \begin{align*}
        \frac{dy}{dz} = -\frac{w_{z}}{w_{y}}\frac{D_1}{D_2}
    \end{align*}
    where $D_1 = z+m_\mu$, $D_2 = y+m_\mu$ are the simplified denominators of $\dot{z}, \dot{y}$. Let $k = w_z/w_y$ and define the function $H(z,y)$ given by 
    \begin{align*}
        H(z,y) &:= \frac{k(1+w_z)z^2}{2} + w_zzy + \frac{1+w_y}{2}y^2,
    \end{align*}
    which is invariant under this ODE, meaning it will be constant along trajectories. As such, by equating the start and end points
    \begin{align*}
        H(0,1) = \frac{1+w_y}{2},\quad H(L,L) = \frac{L^2}{w_y},
    \end{align*}
    we can rearrange to find 
    \begin{align*}
        L(0) = \sqrt\frac{w_y(1+w_y)}{2}.
    \end{align*}
    For $\beta>0$, note that $\lim_{\beta \rightarrow \infty} p_\mu(z,y) = 1$. In this case, the system converges to $(x,y) = (1,1)$, thus $L=1$. The derivative of the non-parametric ODE and the monotonic property of $p_\mu$ in $\beta$  means the limit points are ordered such that $L(\beta_2) \geq L(\beta_1)$ for all $\beta_2>\beta_1$.
\end{proof}
With no selection pressure ($\beta = 0$), the dynamics are exclusively a function of the initial strategy distribution. As such, the limit point can be derived as a function of the initial weights. When this selection term is non-zero, the payoff structure influences the movement of the strategy, and under Assumptions \ref{reward and weight assumption}, the level of cooperation is non-decreasing in $\beta$.

\begin{figure}[H]
     \centering
     \includegraphics[width=\textwidth]{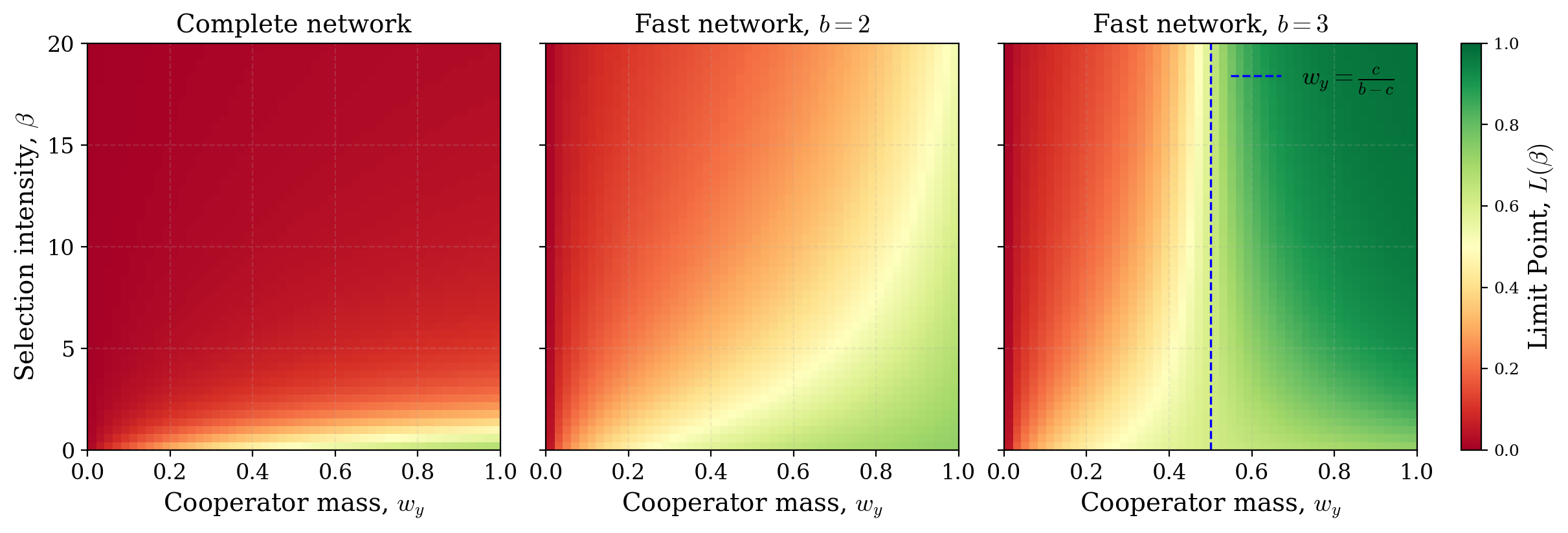}
     \caption{Heat map of the limit point $L(\beta)$ as the weight ($w_y$) of the cooperator cluster varies from 0 (all defectors) to 1 (all cooperators). The payoff $c=1$ is fixed and the effect of changing $b$ near the threshold is shown. The blue dashed line shows the weight threshold for monotonicity in $\beta$ to hold.}
     \label{fig: two type limit point heatmap}
\end{figure}
We have shown that there is a critical point in the payoff parametrisation: $b\geq3c$. In Figure \ref{fig: two type limit point heatmap}, we show that for values of $b$ either side of this transition, the dynamics exhibit distinct behaviour. The complete network in the leftmost figure acts as a benchmark; small levels of cooperation only emerge for very low selection intensities and large cooperative masses. In the centre figure, decreasing the selection pressure causes an increase in the limiting cooperation behaviour. This is a result of the lower selection pressure towards agents with higher rewards, meaning the network structure plays a greater role than the imitation based of payoff. In the right figure, there is a critical transition in the cooperator mass; when sufficient mass is placed in the cooperative cluster, the limit point becomes non-decreasing in $\beta$ and reaches a much higher cooperative level (Proposition \ref{prop: limit point as function of beta}) than for cooperator masses below the threshold. Even with when the majority of the population begins with defection, cooperation can still emerge when the payoff is such that $b\geq3c$. This transition in the weights corresponds exactly to the limits found in Assumption \ref{reward and weight assumption}.

\subsection{Discussion}

We have shown that in both the assortative and complete network dynamics the population converges to a single strategy, with the assortative network inducing a higher rate of cooperation. When starting with two clusters, the cooperation rate increases monotonically with $\beta$ under specified parameter regimes. In particular, we find the payoff threshold $b=3c$ as a pivotal transition in behaviour; after this, the final cooperation rate in the population will increase with the selection intensity parameter $\beta$. This is a result of agents quickly adopting higher payoff strategies, which are inadvertently more cooperative. We also find that there needs to be a critical initial mass of cooperators to ensure cooperation emerges; the exact threshold is dependent on the payoff parametrisation. 

Notably, the population converges to a strategy in the initial convex hull of the population distribution. This is due to a lack of exploration and update variance, which could enable agents to find nearby strategies with higher payoffs. As such, we turn to investigate how noise in the microscopic system alters the dynamics, and in particular the final stationary distribution.

\section{Stochastic microscopic dynamics} \label{Section: With noise}

In this section, we consider how the effects of stochasticity on the evolution and limiting distribution of the model. In agent-based simulations, action selection and the corresponding rewards act as a stochastic process; the expected value is only achieved through an average over infinite interactions. As such, we modify each strategy update with Gaussian noise, and we aim to capture how these individual stochastic effects can change the distribution.

The SDE for the strategy update of an agent is
\begin{align}\label{eqn: agent update SDE}
    dx_i = \frac{1}{ k_i}\sum_{j=1}^Np_{ij}(x_j^2-x_i^2)\; dt + \sigma dW^{(i)}_t,
\end{align}
where $W^{(i)}_t$ is a standard one-dimensional Wiener process, and reflecting boundary conditions are applied. When the policy updates follow an SDE, the population can display interesting dynamics at the cost of additional complexity in solving the system and predicting its behaviour. In particular, agents no longer converges towards a single strategy, yet stable cooperation can still emerge. Figure \ref{fig: finite population graph evolution with noise} shows how with a fast assortative network, the population initially moves towards a near consensus state, before drifting towards cooperation. This exploration of nearby strategies through stochastic noise enables the population to sample outside the convex hull of the current strategy distribution. However, with a complete network this exploration causes the population to drift towards defection. The addition of noise raises the question of how and when cooperation will emerge, and if the resulting distribution will be stable.

\begin{figure}
     \centering
    \includegraphics[width=\textwidth]{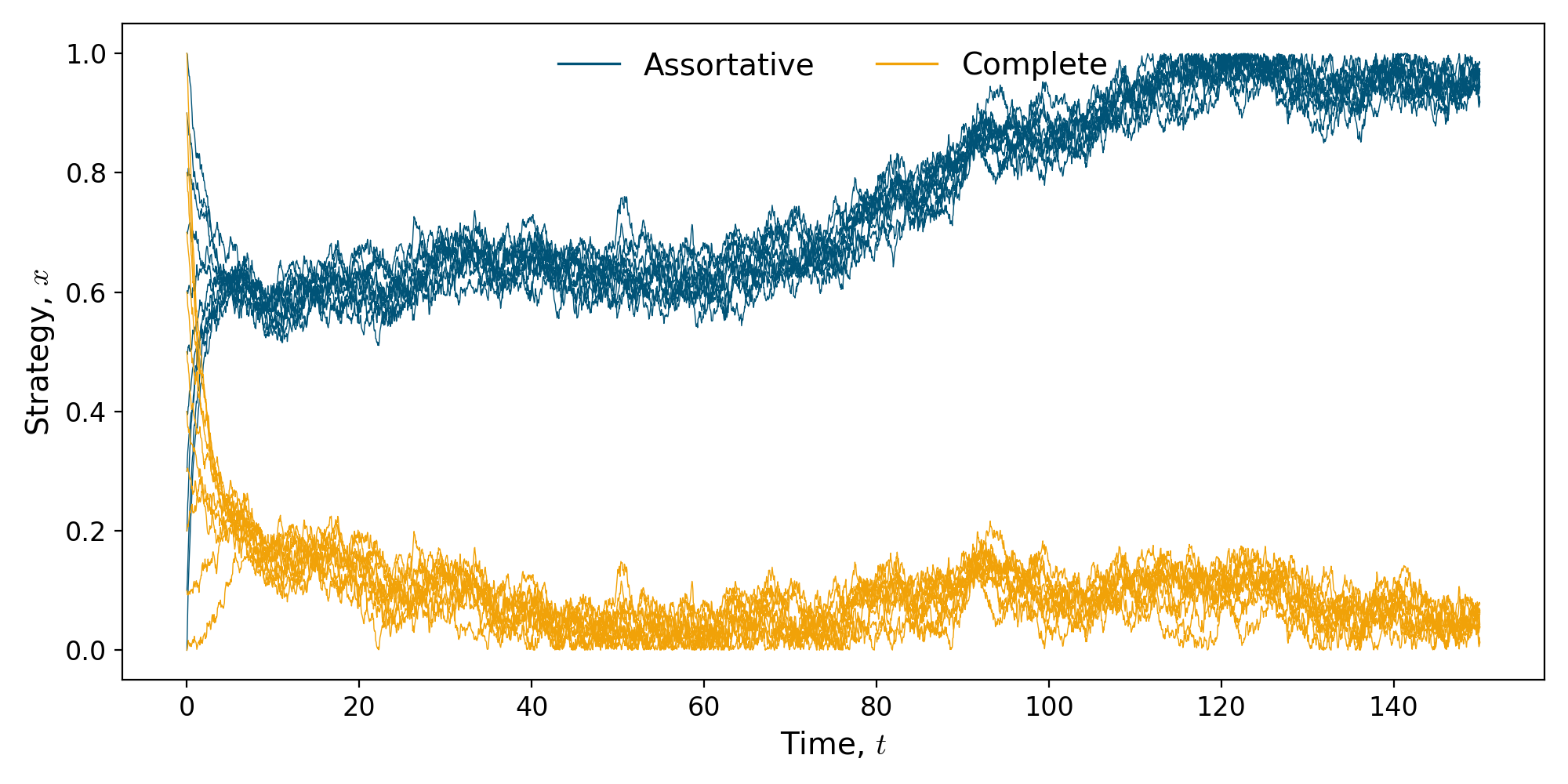}
     \caption{SDE simulation for fast assortative network (blue) and for complete network (orange). This illustrates the how strategy noise can induce higher level of cooperation when the network is assortative. Parameters are $b=4,c=1,\beta =5,N=11, \sigma = 0.03$.}
     \label{fig: finite population graph evolution with noise}
\end{figure}

Noise in the microscopic update means the empirical measure will no longer satisfy the PDE in Equation \eqref{eqn: mean field pde without noise}. The stochastic particle system \eqref{eqn: agent update SDE} is of McKean-Vlasov type, similar to the first order stochastic models discussed in \cite{jabin_meanfieldlimit_2017}. The general derivation of the mean-field limit for these models is provided in \cite{meleard_asymptotic_1996}. For \eqref{eqn: agent update SDE}, the corresponding mean-field limit is then
\begin{align}\label{eqn: mean-field pde}
    \partial_t\mu(t,x) +\partial_x\Bigg( \frac{\mu(t,x)}{x+\int_{\Omega} z\,\mu(t,z)\,dz} \int_{\Omega} (y^2-x^2)\,\mu(t,y)\,p_\mu(x,y)\,dy\Bigg) - \frac{\sigma^2}{2}\partial_{xx}\mu = 0 \,,
\end{align}
We define the flux by
\begin{align*}
    F_\mu(t,x)=\Bigg( \frac{\mu(t,x)}{x+\int_{\Omega} z\,\mu(t,z)\,dz} \int_{\Omega} (y^2-x^2)\,\mu(t,y)\,p_\mu(x,y)\,dy\Bigg) - \frac{\sigma^2}{2}\partial_{x}\mu,
\end{align*}
then the PDE becomes
\begin{align*}
    \partial_t\mu + \partial _xF_\mu = 0.
\end{align*}
The reflecting boundary conditions in the SDE model become no-flux in the mean-field limit. This gives the initial-boundary value problem
\begin{align*}
    \partial_t\mu + \partial _xF_\mu &= 0\\
    F_\mu(t,0)= F_\mu(t,1) &= 0\\
    \mu(0,x) &=\mu_0(x)
\end{align*}

The non-linear and non-local form of the PDE means existence and uniqueness of a solution is an open problem. Since we are more concerned with the long-term dynamics, we instead rigorously prove existence and uniqueness of a stationary distribution under sufficient conditions on the noise parameter, $\sigma^2$. All omitted proofs to results in this section can be found in the Supplementary Material.

\subsection{Numerical solution and properties}
We first present some numerical solutions of the PDE with diffusion. This enables us to analyse how the population evolves given the assortative network structure, and determine if the noise is inhibitor or enabler for the emergence of cooperation.
\begin{figure}
     \centering
    \includegraphics[width=\textwidth]{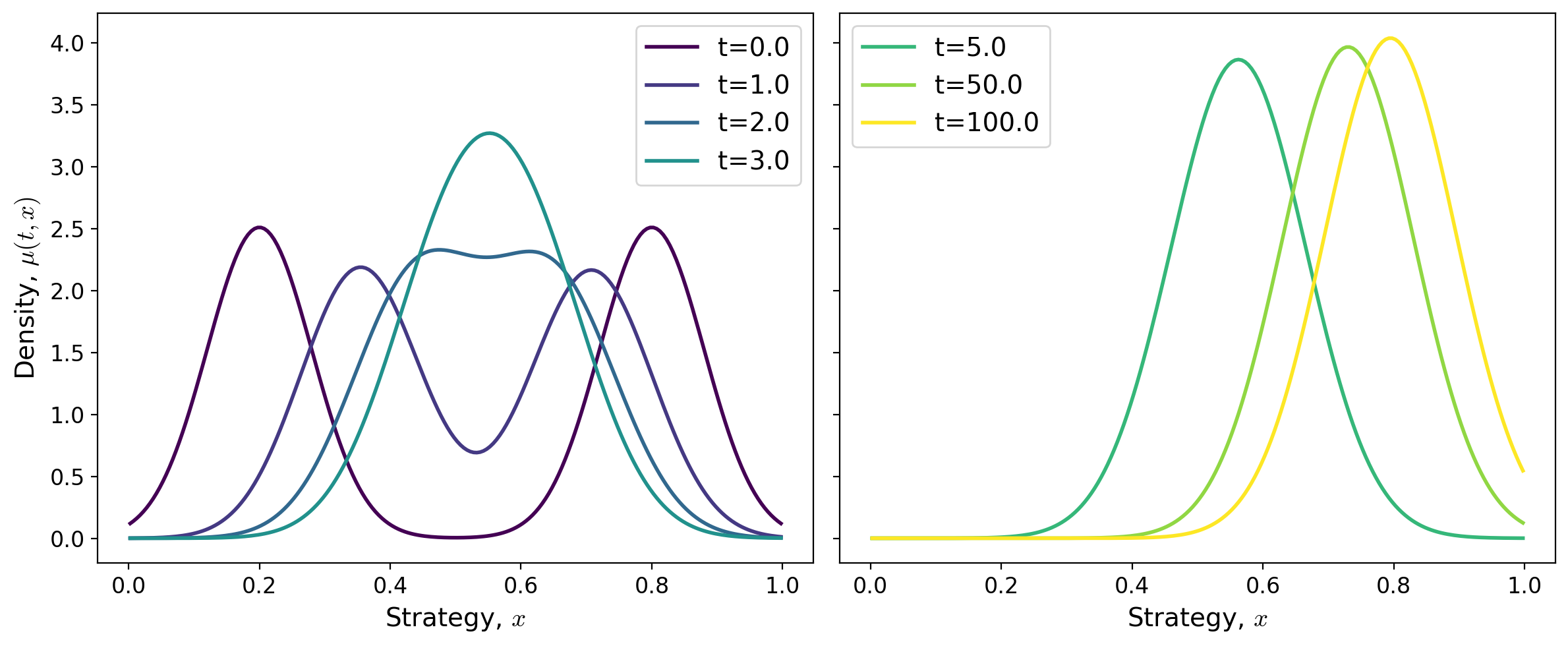}
     \caption{Numerical simulation using the Finite Volume Method of the PDE with an initial bimodal beta distribution. The evolution is separated into two-phases, with strategies combining into a single modal distribution before moving towards cooperation. Parameters are $b=3,c=1,\beta =1, \sigma^2 = 0.01$.}
     \label{fig: PDE noise numerical solution}
\end{figure}
Figure \ref{fig: PDE noise numerical solution} illustrates two distinct phases when the population starts in two clusters, akin to the two-type dynamics discussed in Section \ref{Section: No noise}. In this case, agents no longer concentrate into a single strategy. Instead, the bimodal distribution evolves as follows: i) the two peaks move closer until they combine, ii) the uni-modal distribution becomes more cooperative. These dynamics are explained by the interplay of reward structure and assortative network dynamics. Initially, the cooperative agents are incentivised to reduce cooperation as they can exploit the cooperators they are connected to; defectors increase cooperation to strengthen ties with cooperators. Once the population coalesces, there is little to be gained from defection with there being fewer cooperative agents to exploit. As such, agents can increase their payoff by being well-connected, which is purely a function of their strategy. Consequently, the network dynamics encourage cooperation.

In the system without noise, the population converges to a Dirac (Theorem \ref{thm: convergence of measure to a dirac}) and the support cannot increase from its initial bounds. The addition of noise in the microscopic system allows agents to explore and adopt new strategies in the strategy space which can ultimately lead to more cooperative equilibria. These stochastic effects are concretely present in simulations which show the emergence of cooperative equilibria \cite{leung_learning_2024, defection_russell2026}. These dynamics highlight the importance of a balance between between exploration of nearby strategies and exploitation of those already found. 

\begin{proposition}\label{prop: pde properties}
    Let $\mu(t,x)$ be a strong solution of \eqref{eqn: mean-field pde}, with $\mu(0,x) \geq 0$, then $\mu(t,x) \geq 0,$ and the mass is preserved $\forall t \geq 0$.  
\end{proposition}

The numerical solutions in Figure \ref{fig: PDE noise numerical solution} are computationally expensive and can take a long time to reach their stationary distribution. In the next section, we show how to find the steady state without such calculation, and compare directly to the PDE without diffusion.

\subsection{Steady state analysis}

The stationary distribution captures the long-term dynamics and behaviour of the population. We can analyse the limiting behaviour to observe if and when cooperation is preferable, and if the resulting distribution is stable under perturbations. To find such a stationary distribution, we look for a probability density $\mu(x)$ satisfying the steady state equation Eq.~(\ref{eqn: steady_state}) with boundary conditions (8) and (9)
\begin{subequations}\label{eqn: steady_state}
\begin{align}
    \partial_x\Bigg( \frac{\mu(x)}{x+\int_{\Omega} z\,\mu(z)\,dz} \int_{\Omega} (y^2-x^2)\,\mu(y)\,p_\mu(x,y)\,dy\Bigg) - \frac{\sigma^2}{2}\partial_{xx}\mu &= 0\\
    \mu(0)\Bigg( \frac{1}{\int_{\Omega} z\,\mu(z)\,dz} \int_{\Omega} y^2\,\mu(y)\,p_\mu(0,y)\,dy\Bigg) - \frac{\sigma^2}{2}\partial_{x}\mu(0) &= 0\\
    \mu(1)\Bigg( \frac{1}{1+\int_{\Omega} z\,\mu(z)\,dz} \int_{\Omega} (y^2-1)\,\mu(y)\,p_\mu(1,y)\,dy\Bigg) - \frac{\sigma^2}{2}\partial_{x}\mu(1) &= 0
\end{align}
\end{subequations}

To rigorously show the existence of stationary distributions, we will use the following steps, similar to the approach used in \cite{nugent_opinion_2025}:
\begin{enumerate}
    \item Reduce the condition for stationarity to an ODE.
    \item Introduce a mapping $\mathcal{F}:L^\infty(\Omega) \rightarrow L^\infty(\Omega)$ whose fixed point corresponds to the stationary solution of the ODE, and show that the mean has a non-zero lower bound. 
    \item Find a convex set $K \subset L^\infty(\Omega)$ and show $\mathcal{F}: K \rightarrow K$ is a continuous, compact, and invariant mapping.
    \item Apply Schauder's Fixed Point Theorem. 
\end{enumerate}

For notational convenience and to simplify future calculations, let $\omega = \frac{\sigma^2}{2}$. Due to the no-flux boundary conditions, the steady state occurs when
\begin{align*}
    \partial_xF = F = 0,
\end{align*}
which means we want to find some $\mu$ such that
\begin{align*}
    \frac{\mu}{x+m_\mu}\int_{\Omega}(y^2-x^2)\mu(y)p_\mu(x,y)dy&= \frac{\sigma^2}{2}\partial_x \mu\\
   \Rightarrow \partial_x\mu &=\frac{1}{\omega}\frac{\mu}{x+m_\mu}\int_{\Omega}(y^2-x^2)\mu(y)p_\mu(x,y)dy.
\end{align*}
The spatial derivative of $\mu$ dictates that we look for smooth stationary distributions. This is a natural choice due to the diffusion and observations from numerical simulations. Fixing $\eta\in L^\infty(\Omega)$ as a trial density such that $m_\eta > 0$, define
\begin{align*}
    b_\eta(x) := \frac{1}{\omega}\int_{\Omega}(y^2-x^2)\eta(y)p_{\eta}(x,y)dy,
\end{align*}
then the condition for stationary is the following ODE
\begin{align}\label{eqn: stationary ode}
     \frac{d}{dx}\mu(x) &= \frac{\mu(x)}{x+m_\eta}b_\eta(x).
\end{align}
For a given $\eta\in L^\infty(\Omega)$, let $\mathcal{F}[\eta]$ be the unique solution to Equation \eqref{eqn: stationary ode}, which is given by
\begin{align} \label{eqn: solution to stationary distribution ODE}
\mathcal{F}[\eta](x) &= \frac{w_\eta(x)}{\int_{\Omega} w_\eta(s)ds},\qquad w_\eta(x) = \exp\Big(\int_0^x \frac{b_\eta(s)}{s+m_\eta}ds\Big).
\end{align}

The ODE \eqref{eqn: stationary ode}, and therefore the solution to the PDE, suffers from the same possibility of division by zero discussed in Section \ref{Section: No noise}. In this case, with noise, we consider a critical value, $m^*$, which the mean cannot go below under the fixed point iteration. For a given $m^* \in (0,\frac{1}{2})$, define the set 
\begin{align}\label{set of trial densities K}
    K:=\bigg\{\eta \in L^\infty([0,1]) : \eta\geq0\,, \int^1_0\eta\,dx=1,  m_\eta \geq m^*\bigg\} \,.
\end{align}

The following result puts bounds on the solution to this ODE.
\begin{lemma}\label{lemma: steady state, bounded solution}
    For any $\eta\in L^\infty(\Omega)$ such that $m_\eta > 0$, then the solution to Equation \eqref{eqn: stationary ode}, $\mu = \mathcal{F}[\eta]$, is bounded by
    \begin{align*}
        \mu_0 \bigg(1+\frac{1}{m_\eta}\bigg)^{-\frac{1}{\omega}} \leq \mu(x) \leq \mu_0\bigg(1+\frac{1}{m_\eta}\bigg)^{\frac{1}{\omega}}.
    \end{align*}
\end{lemma}
The updated first moment is given by
\begin{align*}
    m_\mu &= \frac{\int_{\Omega}x \mu(x) dx}{\int_{\Omega}\mu(x)dx}
\geq \frac{\mu_0\big(1+\frac{1}{m_\eta}\big)^{-\frac{1}{\omega}}\int^1_0x dx}{\mu_0\big(1+\frac{1}{m_\eta}\big)^{\frac{1}{\omega}}\int^1_01 dx}
    =\frac{1}{2}\bigg(1+\frac{1}{m_\eta}\bigg)^{-\frac{2}{\omega}}
\end{align*}
To avoid the scenario whereby the mean tends towards zero over a sequence of iterations of this mapping, we need to ensure that the updated mean remains bigger than some critical value $m^*>0$. Formally, define the difference between the updated mean and mean in the previous iteration as
\begin{equation}\label{eqn: g(m,omega)}
        g(m,\omega) = m - \frac{1}{2}(1+\frac{1}{m})^{-\frac{2}{\omega}}.
\end{equation}
Then for a fixed $\omega$, we require that $g(m_\eta,\omega) >0$ to ensure that the mean remains bounded below. The next results provide the necessary and sufficient condition for when a root of $g$ exists.
\begin{lemma}\label{lem: m^* unique}
    There exists a unique solution $m^* \in (0,\frac{1}{2})$ to Equation \eqref{eqn: g(m,omega)} if and only if $\omega > 2$. Moreover, if $m_\eta >m^*$, then $g(m_\eta,\omega)>0$.
\end{lemma}
For this to hold, we require that $\omega>2$, which can be interpreted either as an exploration / exploitation trade-off, or as as a requirement on the level of inherent noise in the system. In particular, this bound is satisfied when there is sufficient noise, or, in other words, when the advection is slow enough relative to the diffusion. In practise, this means exploration of nearby strategies is occurring sufficiently fast to overcome the flow. When $\omega < 2$, no such lower bound $m^*$ exists and it is possible for the mass to accumulate at the pure defective strategy over multiple iterations of the mapping.

\begin{lemma}\label{lem: b is Lipschitz L_infty}
    The function $b_{\eta}(x)$ is Lipschitz continuous in $\eta$ with respect to the $L^\infty$ norm. 
\end{lemma}
\begin{lemma}\label{lem: F well defined and invariant}
    For $\omega>2$, the mapping $\mathcal{F}:K \rightarrow K$ is well defined. Moreover, $\mathcal{F}(K):=\{\mathcal{F}(v): v \in K\} \subset C(\Omega)$. 
\end{lemma}

The following results give further properties of this mapping.
\begin{lemma}\label{lem: F(K) is compact}
    The set $\mathcal{F}(K):=\{\mathcal{F}[\eta]: \eta \in K\} \subset C(\Omega)$ is relatively compact in $L^\infty(\Omega)$. 
\end{lemma}

\begin{proposition}\label{prop: F_bounded}
    The mapping $\mathcal{F}:K \rightarrow K$ is uniformly bounded in $K$. That is, there exists some constant $M_\mathcal{F} \in \mathbb{R^+}$ such that for all $\eta \in K$, \begin{equation*}
        \|\mathcal{F}[\eta]\|_{\infty}\leq M_\mathcal{F}
    \end{equation*} 
\end{proposition}

\begin{proposition}\label{prop: F operator is continuous}
    The mapping $\mathcal{F}:K \rightarrow K$ is Lipschitz continuous with respect to supremum norm on $C(\Omega)$. Specifically, there exists a constant $L_\mathcal{F}$ which depends upon $\omega$ and $m^*$, such that
    \begin{align*}
        \|\mathcal{F}[\eta]-\mathcal{F}[\nu]\|_\infty &\leq L_\mathcal{F}\|\eta-\nu\|_\infty.
    \end{align*}
\end{proposition}

\begin{theorem}
    There exists at least one solution $\mu \in K$ to the steady state equation \eqref{eqn: steady_state}. 
\end{theorem}
\begin{proof}
    Let $K$ be defined in \eqref{set of trial densities K}. Then $K$ is a non-empty, closed and convex set in $L^\infty(\Omega)$. A solution to the steady state equation is uniquely expressed by the fixed point of the mapping $\mathcal{F}:K\rightarrow K$, defined by the solution to \eqref{eqn: stationary ode}. By Lemma \ref{lem: F well defined and invariant}, the mapping is well defined and invariant. By Proposition \ref{prop: F operator is continuous}, $\mathcal{F}$ is continuous on $K$, and by Lemma \ref{lem: F(K) is compact} the image $\mathcal{F}(K)$ is compact in $L^\infty(\Omega)$. Applying Schauder's Fixed Point Theorem (see e.g. Corollary 7.4 of \cite{shapiro_fixed_point}) to $\mathcal{F}:K\rightarrow K$ gives the existence of at least one fixed point of $\mathcal{F}$ in $K$, and therefore existence of at least one solution to the steady state equation.
\end{proof}

We next address the question of uniqueness. In this case, it is sufficient to show that the Lipschitz constant for the operator $\mathcal{F}$ is less than one, creating a contraction. By exploiting the relationship between $m^*$ and $\omega$ through Equation \eqref{eqn: g(m,omega)}, we can reduce the dependency to just one variable and prove monotonic behaviour.

\begin{lemma}\label{lem: L_F is monotone decreasing in m}
    Let $L_\mathcal{F}$ be the Lipschitz constant of the mapping $\mathcal{F}$,  defined in Proposition \ref{prop: F operator is continuous}. Then $L_\mathcal{F}(m^*)$ is monotone decreasing in $m^*$.
\end{lemma}

\begin{lemma}\label{lem: m(omega) is monotone increasing in omega}
    For $m>0,\omega>2$, the solution $m^*(\omega)$ to $g(m,\omega)=0$ is monotone increasing in $\omega$.
\end{lemma}

Since the function $L_\mathcal{F}$ is continuous and monotone decreasing in $m^*$, there exists a critical value $\omega_{crit}$ and $m'$ such that $L_\mathcal{F}(m') =1$ and $\forall m^* > m', L_\mathcal{F}(m^*) < 1$. 
\begin{theorem}\label{thm: unique fixed point}
    There exists an $\omega^*>2$ such that for all $\omega > \omega^*$, there is a unique solution to the steady state equation \eqref{eqn: steady_state}. 
\end{theorem}
We next verify that the uniform distribution cannot be a stationary distribution.
\begin{proposition}\label{prop: uniform is not stationary}
    The uniform distribution is not stationary.
\end{proposition}

Here, we extend the limit point analysis in Section \ref{Section: No noise}, providing an explicit approximation of the fixed point which explains the observed dynamics. 

Since the fixed point must satisfy
\begin{align*}
    \mu(x) \propto \exp\bigg(\int_0^x \frac{b_\mu(s)}{s+m_\mu} ds\bigg),\quad  m_\mu = \int_\Omega x \mu(x) dx,
\end{align*}
we have
\begin{align*}
    \frac{d}{dx}\log\mu(x) = \frac{b_\mu}{x+m_\mu}.
\end{align*}
We assume that the fixed point can be approximated by a unimodal Gaussian; whilst this cannot hold on the domain $\Omega= [0,1]$, the simulated results demonstrate a clear uni-modal strategy with symmetric variance for larger values of $\beta$. At the peak of the distribution, $\hat{x}$, this differential equation is exactly zero. Therefore, it must satisfy
\begin{align*}
    b_\mu(\hat{x})&=0,\\
    \int_\Omega (y^2-\hat{x}^2)\mu(y)p(\hat{x},y) dy &= 0.
\end{align*}
To approximate the Fermi function $p_\mu$, we assume that the mass is concentrated near a singular point such that $m_\mu =\hat{x}$, which requires a sufficiently small variance and skew. The opponents strategy $y$ is a random perturbations away from the mean such that 
\begin{align*}
    y = m_\mu + \xi,\quad  \xi \sim N(0,v^2)
\end{align*}
which reduces the argument of $p_\mu$ to
\begin{align*}
    \beta(y-\hat{x})[(b-c)m_\mu - c(y+\hat{x})] &= \beta \xi [(b-c)m_\mu -c(2m_\mu+\xi)]\\
    &=\beta\xi(b-3c)m_\mu - c\beta\xi^2.
\end{align*}
Ignoring terms of order $\xi^2$ and using a Taylor expansion of the soft-max function, the condition for the peak becomes
\begin{align*}
    \int_\Omega (2m_\mu\xi + \xi^2)\mu(y)(\frac{1}{2} + \frac{1}{4}\beta\xi(b-3c)m_\mu) dy &= 0\\
    \Leftrightarrow\qquad \mathbb{E}[(2m_\mu\xi + \xi^2)(\frac{1}{2} + \frac{1}{4}\beta\xi(b-3c)m_\mu)] &= 0\\
    \Leftrightarrow\qquad v^2(\frac{1}{2} + \frac{1}{2}\beta(b-3c)m^2_\mu ) &=0\\
   \Rightarrow\qquad m_\mu &= \frac{1}{\sqrt{\beta(3c-b)}}
\end{align*}
which holds for $b<3c$.

This approximation provides clarity on how the distribution is affected by the selection intensity. Higher values of $\beta$ force a hard selection, whereby agents move very quickly towards agents who have a higher payoff. As a result, the fixed points support a tighter distribution of strategies, resulting in a uni-modal distribution. The particular threshold of $b=3c$ emerges again, as in the two-type dynamics discussed in Section \ref{Section: No noise}. This payoff ratio emerges as the key determinant of defective or cooperative outcomes when the selection pressure is increased. When the threshold is not met, small values of $\beta$ can enable more cooperative outcomes, as the strategy evolution depends less on the payoff difference and more on the assortative network structure. 

Figure \ref{fig: fixed point heat map} shows that the approximation qualitatively captures the dynamics when the selection pressure $\beta$ is varied for $b<3c$. Higher values of $\beta$ increase the accuracy of the approximation, suggesting that the distribution is better approximated by a Gaussian. The approximation does not hold for $b\geq 3c$ at first or second-order, due to the sign of the leading correction term. In this regime, the distribution centres near $0.8$ regardless of the selection intensity. In the case of a complete static network, the distribution clusters near pure defection ($x=0$) for all values of $b$ and $\beta$. The diffusion constant $\sigma^2$ does not meet the sufficient condition for a well-definition mapping ($\omega>2$), however we still find the population avoid collapse to pure defection.
\begin{figure}
     \centering
    \includegraphics[width=\textwidth]{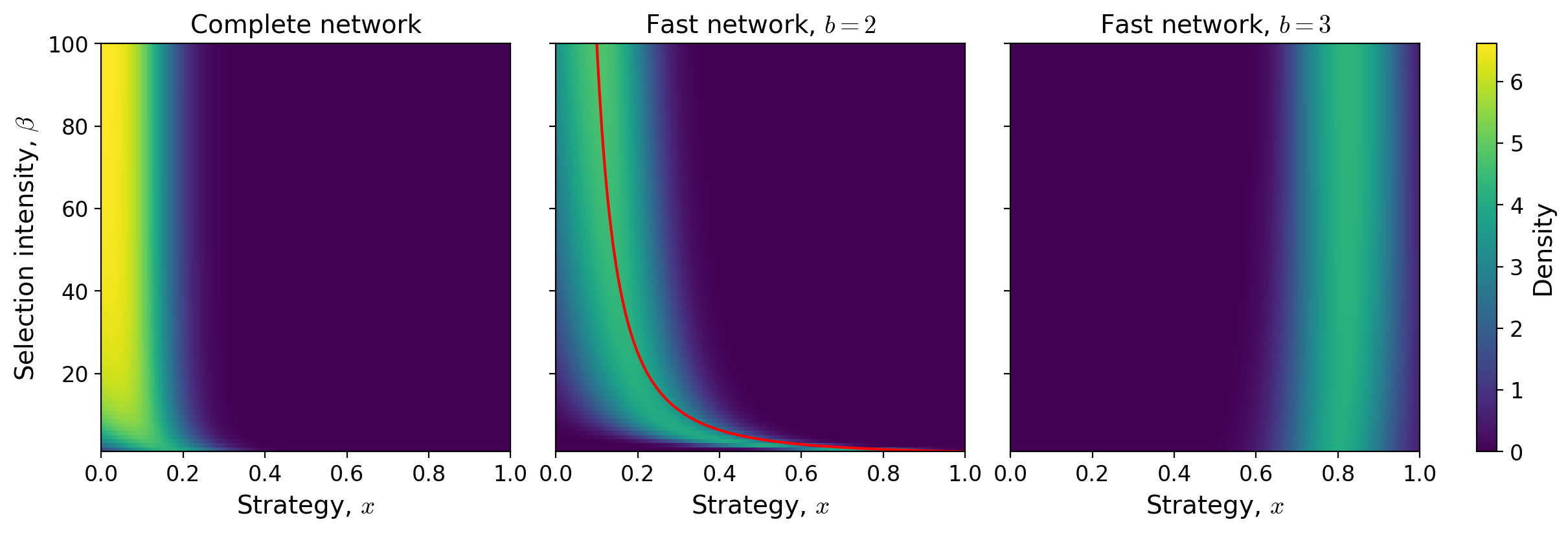}
     \caption{Fixed point of the mapping $\mathcal{F}$ for different values of $\beta$ near the threshold $b=3c$. In the centre figure, the approximated peak of the distribution is plotted in red; the accuracy improves for larger values of $\beta$. In a complete network, the population clusters around defection. For the fast assortative network, when $b=2$, increasing the selection pressure induces more defection; when $b=3$ the network dynamics shift the population towards a cooperative distribution. The payoff $c=1$ is fixed and $\sigma^2 = 0.01$.}
     \label{fig: fixed point heat map}
\end{figure}

\subsection{Linear stability analysis}
A central question in evolutionary game theory, and more broadly in the study of dynamical systems, is whether a limiting distribution is stable under small perturbations. In the previous section, we identified parameter regimes in which the dynamics converge to more cooperative stationary distributions. We now show that such equilibria are stable: if the stationary distribution is slightly perturbed, then the perturbation remains controlled, and in a high-diffusion regime it decays in time.  

To do this, let $\mu^*$ be the fixed point to Equation \ref{eqn: mean-field pde}, and consider a mass-zero perturbation of size $\varepsilon$ in the direction of $\eta$, such that the distribution $\mu = \mu^* + \varepsilon \eta $. Substituting this ansatz into the non-linear equation and keeping only first order terms in $\varepsilon$ gives the linearised dynamics about $\mu^*$. The perturbation $\eta$ solves
\begin{subequations}\label{eqn: linear stability pde}
\begin{align}
    \partial_t \eta(t,x) + \partial_x \mathcal{L}[\eta](x)- \frac{\sigma^2}{2}\partial_{xx}\eta(t,x) &= 0\\
    \mathcal{L}[\eta](x) -\frac{\sigma^2}{2}\partial_x\eta(t,x)\bigg|_{x=0,1} &= 0\\
    \eta(0,x) &= \eta_0(x)
\end{align}
\end{subequations}
where the linear operator $\mathcal{L}[\eta](x)$ is
{
\small
\begin{align}\label{eqn: operator L}
 =\eta V_{\mu^*} + \frac{\mu^*}{{x+m_{\mu^*}} }\int_\Omega(y^2-x^2)[\beta(b-c)m_\eta(y-x)\mu^*p_{\mu^*}(1-p_{\mu^*}) + \eta p_{\mu^*}]dy - \frac{\mu^*I_{\mu^*} m_\eta}{(x+m_{\mu^*})^2}.
\end{align}
}
The derivation of this linearisation is provided in the Supplementary Material. This is a non-local parabolic problem, where the diffusion term is local whilst the operator $\mathcal{L}$ couples the perturbation at a point to the global moments of the distribution. We will show that this linear problem is well-posed, before using an energy estimate to show that perturbations decay for sufficiently large diffusion. The following result bounds the operator $\mathcal{L}$, which is necessary to treat the problem in a standard framework.
\begin{proposition}\label{prop: linear stability, operator L is bounded}
    The operator $\mathcal{L}[\eta]$ is bounded in the $L_2$ norm such that 
    \begin{align*}
        \|\mathcal{L}[\eta](x)\|_2 \leq M_\mathcal{L}\|\eta\|_2
    \end{align*}
    for some $M_\mathcal{L} \in \mathbb{R}$.
\end{proposition}
As in the non-linear problem (Proposition \ref{prop: pde properties}), the no-flux boundary conditions imply conservation of mass. If this initial perturbation has zero mass, then this property is preserved for all later times. This will be important when we apply the Poincaré-Wirtinger inequality in the decay estimate.

To form the problem in the weak sense, define $U=H^1(\Omega):= \{f \in L^2(\Omega) : \nabla f \in L^2(\Omega)\}$, and $H = L^2(0,1)$. We then introduce the bilinear form $a: U \times U \rightarrow \mathbb{R}$ as
\begin{align}\label{eqn: a(u,psi)}
    a(u,\varphi) := \frac{\sigma^2}{2}\int_\Omega u_x \varphi_x dx - \int_\Omega \mathcal{L}[u](x) \varphi_x(x) dx.
\end{align}
This is the natural form associated with the linearised PDE, obtained by integrating against a test function $\varphi \in U$ and applying integration by parts.
\begin{proposition}\label{prop: linear stability, a is bilinear}
    The mapping $a(\cdot,\cdot)$ is a continuous bilinear form on $U$, and satisfies
    \begin{itemize}
        \item $|a(\eta, \varphi)| \leq c\;\|\eta\|_U\|\varphi\|_U,\quad c \in \mathbb{R}$ 
        \item $\exists \lambda \in \mathbb{R}$, $a(\eta,\eta) + \lambda \|\eta\|_{H}^2 \geq \alpha \|\eta\|_U^2$, where $\alpha >0, \forall \eta \in U$.
    \end{itemize}
\end{proposition}
The second property is Gaarding-type inequality, and combined with continuity allows us to apply the standard Lions-Magenes \cite{LionsMagenes1972} theory for linear evolution equations. The anti-linear form $v\rightarrow a(u,v)$ is continuous on $U$, hence 
\begin{align*}
    a(u,v) = \langle \mathcal{A}u,v\rangle_{U',U},
\end{align*}
where $\mathcal{A}u \in U'$ and $\mathcal{A}$ is a linear operator. Note that $\langle \cdot, \cdot \rangle$ denotes the duality pairing, between $U'$ and $U$.

To derive the weak formulation, test \eqref{eqn: linear stability pde} against $\varphi \in U$:
\begin{align*}
    \int_\Omega \partial_t \eta(t,x)\varphi(x)\; dx + \int_\Omega \partial_x\mathcal{L}[\eta](x) \varphi(x) \; dx - \frac{\sigma^2}{2} \int_\Omega \partial_{xx} \eta(t,x) \varphi(x)\; dx &= 0.
\end{align*}
Using integration by parts, the left hand side becomes
\begin{align*}
    \langle \partial_t\eta, \varphi\rangle + \frac{\sigma^2}{2}\int_\Omega \partial_x \eta \partial_x \varphi\; dx - \int_\Omega \mathcal{L}[\eta](x) \partial_x\varphi(x)\; dx+ \Big[(\mathcal{L}[\eta](x) - \frac{\sigma^2}{2} \partial_x 
    \eta)\varphi(x)\Big]_{\partial\Omega}.
\end{align*}
The no-flux boundary conditions of the PDE imply the final term is zero and the weak form reduces to
\begin{align*}
    \langle \partial_t\eta, \varphi\rangle + a(\eta,\varphi) &= 0,\\
    \iff \langle \partial_t\eta, \varphi\rangle  + \langle \mathcal{A}\eta,\varphi\rangle &= 0,\\
    \iff \langle \partial_t\eta + \mathcal{A}\eta,\varphi\rangle &= 0.
\end{align*}
Since this must hold for all test-functions $\varphi \in U$,
\begin{align*}
    \partial_t\eta + \mathcal{A}\eta &= 0.
\end{align*}

\begin{theorem}\label{thm: solution to linear pertubation equation}
    Let $\mathcal{L}$ be the operator defined by Equation \eqref{eqn: operator L}, and $\mathcal{A} \in \mathcal{L}(U,U')$ be the linear operator associated with the bilinear form \eqref{eqn: a(u,psi)}. For every $\eta_0 \in H$, there exists a unique $\eta \in L^2(0,T;U)\; \cap\; H^1(0,T;U')$ such that $$\partial_t\eta + \mathcal{A}\eta = 0,\quad \eta(0) = \eta_0.$$
    That is, there is a unique weak solution to the linear perturbation equation.
\end{theorem}

We now can consider the long-term dynamics of the solutions: when diffusion is sufficiently strong relative to the non-local advection, the equilibrium is linearly stable in the sense that every mean-zero perturbation decays exponentially fast.
\begin{proposition}\label{prop: decaying perturbation}
    Let $\eta_0 \in L^2(\Omega)$ such that $\int_\Omega \eta_0 \, dx = 0$, and let $\eta \in L^2 \cap H^1$ be the unique weak solution to \eqref{eqn: linear stability pde}. If
    \begin{align*}
        \sigma^2 > 2CM_{\mathcal{L}},
    \end{align*}
    where $C$ is the Poincaré–Wirtinger constant, then
    \begin{align*}
        \|\eta(t)\|^2_2 \leq e^{-\frac{\gamma}{C^2}t} \|\eta_0\|_2^2.
    \end{align*}
    for some $\gamma>0$. That is, the perturbation exponentially decays to zero in $L^2(\Omega)$.
\end{proposition}

This shows how noise in the strategy updates can enable stochastic stability of cooperative equilibria, with the unknown constant acting as a sufficient lower bound. As discussed in the deterministic case (Section \ref{Section: No noise}), the Dirac is inherently unstable with any perturbation changing the limiting distribution. Noise not only enables the population to converge to a more cooperative population but also sustain this level. Agents do not discover defection despite exploration; the fast network pulls defectors towards the mean strategy by reducing their edge weights with cooperative individuals.

\subsection{Discussion}

The addition of noise changes the evolution and stationary distribution compared with the deterministic dynamics. Without noise the population moves to consensus, with the final strategy heavily constrained by the initial strategy support. Diffusion competes with the imitation force, and enables a smooth distribution to emerge; this is a non-uniform distribution, with the modal strategy capturing the balance between exploration and exploitation. Compared with a complete network, the fast assortative network promotes cooperation by rewarding cooperative agents with stronger connections with cooperators. Moreover, the addition of noise induces linear stability of the stationary distribution; in the deterministic dynamics the final strategy is sensitive to perturbations.

For existence of stationary distributions, we find a sufficient condition of the noise, $\omega>2$. This prevents collapse to a the pure defection strategy by bounding the mean away from zero. In practise, much lower values of $\omega$ do not cause the mass accumulation at zero suggesting that only some stochastic exploration is required.

The payoff threshold $b=3c$ can be interpreted in line with other threshold values. In \cite{ohtsuki_simple_2006}, a different model on a $k$-regular graph with binary strategies and weak-selection is studied; the authors find a similar threshold, with cooperation  favoured when $b>kc$. This threshold proves sufficient for cooperation in static networks of human populations \cite{rand_static_2014}. Here, the threshold emerges from the fast assortative weight dynamics rather than a fixed graph degree. 

\section{Conclusion}\label{Section: Conclusion}

We have studied imitation dynamics for a continuous strategy Prisoner's Dilemma on a fast adaptive network. In this fast-slow system, the assortative dynamics reduce to an explicit interaction kernel dependent only upon the strategies of the two agents. As a result, we can formally derive the mean-field large population limit from the finite agent model.

In the deterministic setting this yields a non-local continuity equation. We proved well-posedness and uniqueness on any finite time interval, and showed convergence to a Dirac mass as $t \rightarrow \infty$. Without strategy exploration, the population converges to consensus. The analysis reveals how the fast network structure induces a more cooperative final strategy than the complete static network. The two interaction structures can be viewed as limits of the same initial complete graph: the assortative network occurs when the network evolution is much faster than the strategy evolution ($\tau\rightarrow 0$) and the complete network is equivalent to much faster strategy evolution than network ($\tau \rightarrow \infty$). This highlights how the relative timescale between strategy and network evolution is pivotal to the emergence of cooperation. 

Adding stochasticity in the microscopic model changes the macroscopic dynamics, leading to a non-local Fokker-Planck equation. Rather than collapsing to a Dirac measure, the population can reach a smooth and linearly stable stationary distribution. Numerics show how the stochastic dynamics cause separate strategies to cluster, before drifting towards a cooperative equilibria. This marks a clear shift when compared to the deterministic dynamics; the addition of noise can induce a higher level of cooperation through exploration of nearby strategies. In the assortative case, this exploration enables agents to increase their cooperation level and strengthen interactions with other cooperates. In contrast, exploration in the complete static network enables defection to take hold of the whole population. This switch shows how exploration is beneficial only if mechanisms exists to support pro-social behaviour.

Whether cooperation emerges or fails to materialise is subject to the exact payoff parametrisation. We find that the payoff threshold of $b=3c$ is pivotal when the microscopic updates are deterministic and stochastic. When starting with only cooperators and defectors, cooperation can emerge despite starting in the minority. With diffusion, the population supports a high level of cooperation when this threshold is met. This threshold only emerges under the fast-assortative network; with a complete network the payoff ratio has no impact. This highlights how the choice of interaction structure is key in finding the conditions for the emergence of stable cooperation. 

It is clear from this work that evolving networks and strategy heterogeneity in the form of stochastic exploration can support cooperative behaviour in populations of self-interested agents. This analysis could be extended beyond the extreme-popularity linking rule to other adaptive network mechanics, whilst the non-linear stability and well-posedness of the Fokker-Planck remain open questions. We hope future work can build upon our analysis to further explain and promote such macroscopic phenomena that arises from microscopic interactions.

\begin{acks}
BR and AN were supported by the Engineering and Physical Sciences Research Council through the Mathematics of Systems II Centre for Doctoral Training at the University of Warwick (reference EP/S022244/1). JB acknowledges the support of the Cooperative AI Foundation. For the purpose of open access, the author has applied a Creative Commons Attribution (CC BY) licence to any Author Accepted Manuscript version arising from this submission.
\end{acks}

\section*{CRediT author statement}
\noindent 

Benedict Russell: Conceptualisation, Formal Analysis, Investigation, Methodology, Software, Visualisation, Writing - Original Draft, Writing - Review \& Editing.

Andrew Nugent: Conceptualisation, Methodology, Supervision, Writing - Review \& Editing.

Jacques Bara: Conceptualisation, Supervision, Writing - Review \& Editing.

\printbibliography
\newpage
\appendix
\section{Additional derivations}\label{appendix: additional}
The full mean-field derivation reads
\begin{align*}
    &\int_\Omega \partial_t\mu(t,x)\varphi(x)dx\\ &= \frac{d}{dt}\int_\Omega\frac{1}{N}\sum_{i=1}^N \delta_{x_i}(x) \varphi(x)dx\\
    &=\frac{d}{dt}\frac{1}{N}\sum_{i=1}^N\varphi(x_i)\\
    &=\frac{1}{N} \sum_{i=1}^N \varphi'(x_i) \frac{dx_i}{dt}\\
    &=\frac{1}{N} \sum_{i=1}^N \varphi'(x_i) \Bigg[\frac{1}{k_i}\sum_{j=1}^N a_{ij}p_{\underline{x}}(x_i,x_j)(x_j-x_i)\Bigg]\\
    &=\frac{1}{N} \sum_{i=1}^N \varphi'(x_i) \Bigg[\frac{1}{\sum_{k=1}x_i+x_k}\sum_{j=1}^N (x_i+x_j)p_{\underline{x}}(x_i,x_j)(x_j-x_i)\Bigg]\\
    &=\int_\Omega\frac{1}{N} \sum_{i=1}^N \varphi'(x) \Bigg[\frac{1}{\sum_{k=1}x+x_k}\sum_{j=1}^N (x_j^2-x^2)p_{\underline{x}}(x,x_j) \Bigg] \,\delta_{x_i}(x) \, dx\\
    &=\int_\Omega \Bigg[ \frac{1}{N} \sum_{i=1}^N \,\delta_{x_i}(x) \Bigg] \, \varphi'(x) \,\Bigg[ \frac{1}{\sum_{k=1}x+x_k}\sum_{j=1}^N (x_j^2-x^2)p_{\underline{x}}(x,x_j) \Bigg] \, dx\\
    &=\int_\Omega\mu(t,x) \,\varphi'(x) \Bigg[\frac{1}{\sum_{k=1}x+x_k}\sum_{j=1}^N (x_j^2-x^2)p_{\underline{x}}(x,x_j)\Bigg] dx\\
    &=\int_\Omega\mu(t,x) \,\varphi'(x) \Bigg[\int_\Omega \frac{1}{\sum_{k=1}x+x_k}\sum_{j=1}^N (y^2-x^2)p_{\underline{x}}(x,y)\,\delta_{x_j}(y)dy\Bigg] dx\\
    &=\int_\Omega\mu(t,x) \,\varphi'(x) \Bigg[ \frac{1}{\frac{1}{N}\sum_{k=1}x+x_k} \int(y^2-x^2)\mu(t,y)p_{\underline{x}}(x,y)dy\Bigg] dx\\
    &=\int_\Omega\mu(t,x) \,\varphi'(x) \Bigg[ \frac{1}{(x+\int_\Omega z\mu(t,z)dz)} \int(y^2-x^2)\mu(t,y)p_{\underline{x}}(x,y)dy\Bigg] dx\\
    &=-\int_\Omega \,\varphi(x) \, \partial_x \Bigg( \mu(t,x)\Bigg[ \frac{1}{(x+\int_\Omega z\mu(t,z)dz)} \int(y^2-x^2)\mu(t,y)p_{\underline{x}}(x,y)dy\Bigg] \Bigg) dx.
\end{align*}

For a complete network, $k_i=N$ for each agent. Hence we get
\begin{align*}
    \int_\Omega \partial_t\mu(t,x)\varphi(x)dx &= \frac{d}{dt}\int_\Omega\frac{1}{N}\sum_{i=1}^N \delta_{x_i}(x) \varphi(x)dx\\
    &=\frac{1}{N} \sum_{i=1}^N \varphi'(x_i) \Bigg[\frac{1}{N}\sum_{j=1}^N p_{\underline{x}}(x_i,x_j)(x_j-x_i)\Bigg]\\
    &=\int_\Omega\frac{1}{N} \sum_{i=1}^N \varphi'(x) \Bigg[\frac{1}{N}\sum_{j=1}^N p_{\underline{x}}(x,x_j)(x_j-x) \Bigg] \,\delta_{x_i}(x) \, dx\\
    &=\int_\Omega \Bigg[ \frac{1}{N} \sum_{i=1}^N \,\delta_{x_i}(x) \Bigg] \, \varphi'(x) \,\Bigg[\frac{1}{N}\sum_{j=1}^N (x_j-x)p_{\underline{x}}(x,x_j) \Bigg] \, dx\\
    &=\int_\Omega\mu(t,x) \,\varphi'(x) \Bigg[\frac{1}{N}\sum_{j=1}^N (x_j-x)p_{\underline{x}}(x,x_j)\Bigg] dx\\
    &=\int_\Omega\mu(t,x) \,\varphi'(x) \Bigg[\int_\Omega \frac{1}{N}\sum_{j=1}^N (y-x)p_{\underline{x}}(x,y)\,\delta_{x_j}(y)dy\Bigg] dx\\
    &=\int_\Omega\mu(t,x) \,\varphi'(x) \Bigg[\int(y-x)\mu(t,y)p_{\underline{x}}(x,y)dy\Bigg] dx\\
    &=-\int_\Omega \,\varphi(x) \, \partial_x \Bigg( \mu(t,x)\Bigg[ \int(y-x)\mu(t,y)p_{\underline{x}}(x,y)dy\Bigg] \Bigg) dx.
\end{align*}

For the selection term $p_\mu$, first consider the payoff
\begin{align*}
    \pi_i = b\bigg(\sum_{j=1}^N x_j \bigg) - cN x_i.
\end{align*}
Therefore, the payoff difference is
\begin{align*}
    \pi_j-\pi_i =  -cN(x_j-x_i),
\end{align*}
and finite population Fermi function is
\begin{align*}
    p_{\underline{x}}(x_i,x_j) &= \frac{1}{1+\exp(cN\beta(x_j-x_i))}.
\end{align*}
Using the same mean-field scaling $\beta= \beta \times 2/N$, we formally have
\begin{align*}
    p_{\mu}(x,y) &= \frac{1}{1+\exp(2\beta c(y-x))}.
\end{align*}

\section{Proofs: Section \ref{Section: No noise}} \label{Appendix: No noise}

\subsubsection*{Proof of Lemma \ref{lem: mean bounded below}}
\begin{proof}
    We bound the drift term below,
    \begin{align*}
        I(x,\mu) \geq \int_{\Omega} -x^2\,\mu(t,y)\,dy \geq -x^2,
    \end{align*}
    which means
    \begin{align*}
        V_\mu(t,x) \geq -\frac{x^2}{x+m_{\mu}(t)} \geq -x.
    \end{align*}
    Now consider the time derivative of the mean, $m_\mu$:
    \begin{align*}
        \frac{dm}{dt} &= \frac{d}{dt} \int_\Omega z\mu(t,z) dz\\
        &=  \int_\Omega z\partial_t\mu(t,z) dz\\
        &=  -\int_\Omega z  \partial_z (\mu(t,z) V_\mu(t,z) dz\\
        &=  \int_\Omega \mu(t,z) V_\mu(t,z) dz\\
        &\geq - \int_\Omega z\mu(t,z) dz\\
        &= -m_\mu(t)
    \end{align*}
    Therefore
    \begin{align*}
        m_\mu(t) \geq m_\mu(0)e^{-t},
    \end{align*}
    hence the mean remains bounded for any finite time.
\end{proof}

\subsubsection*{Proof of Lemma \ref{lem: p is Lipschitz W_1}}
\begin{proof}
We will begin by showing the difference means is Lipschitz with respect to $W_1(\eta,\nu)$. By the Kantorovich--Rubinstein duality, 
\begin{align*}
    |m_{\eta}-m_{\nu}| &= \Big|\int_\Omega xd(\eta-\nu)(x)\Big| \leq W_1(\eta,\nu)
\end{align*}
We will first prove continuity in $\eta$. Let $Q_\eta = -\beta(y-x)\Big((b-c)m_\eta-c(y+x)\Big)$, then
\begin{align*}
    |p_{\eta}(x,y) - p_{\nu}(x,y)| &= \bigg|\frac{1}{1+\exp\big(Q_\eta\big)}-\frac{1}{1+\exp\big(Q_\nu\big)}\bigg|\\
    &= \frac{\Big|\exp\big(Q_\eta\big)- \exp\big(Q_\nu\big)\Big|}{\Big|\Big(1+\exp\big(Q_\eta\big)\Big)\Big(1+ \exp\big(Q_\nu\big)\Big)\Big|}\\
    &\leq \frac{\Big|Q_\eta-Q_\nu\Big|\exp({\max{\{Q_\eta,Q_\nu\}}})}{\Big|1+\exp\big(Q_\eta\big)\Big|\Big|1+ \exp\big(Q_\nu\big)\Big|}\\
    &\leq \frac{\Big|Q_\eta-Q_\nu\Big|\Big|1+\exp({\max{\{Q_\eta,Q_\nu\}}})\Big|}{\Big|1+\exp\big(Q_\eta\big)\Big|\Big|1+ \exp\big(Q_\nu\big)\Big|}\\
    &\leq \Big|Q_\eta-Q_\nu\Big|\\
    &= \Big|\beta(y-x)(b-c)(m_{\eta}(t) - m_{\nu}(t))\Big|\\
    &\leq \beta (b-c)\Big|m_{\eta}(t) - m_{\nu}(t)\Big|\\
    &\leq \beta (b-c)W_1(\eta,\nu)
\end{align*}

We now prove that $p_\eta(x,y)$ is Lipschitz in $x$ (proof is equivalent for $y$). First note that the sigmoid function $\sigma(r)$ has derivative $\sigma'(r) = \sigma(r)(1-\sigma(r))$. On the domain of $r \in \mathbb{R}$, this attains a maximum of 1/4. Therefore, applying mean-value theorem and the chain rule
     \begin{align*}
         |p_\eta(x_1,y) - p_\eta(x_2,y)| &\leq \frac{1}{4}\beta(b +c)| x_1 - x_2|,
     \end{align*}
and therefore $p_\mu$ is Lipschitz in each argument.
\end{proof}

\subsubsection*{Proof of Lemma \ref{lem: V is Lipschitz W_1}}
\begin{proof}
    We begin by showing the integral term is Lipschitz. For any $x\in \Omega$,
\begin{align*}
    |I_{\eta}(t,x) - I_{\nu}(t,x)|&= \Big| \int_\Omega(y^2-x^2)p_{\eta}(x,y)\;d\eta(y) - \int_\Omega(y^2-x^2)p_{\nu}(x,y)\;d\nu(y)\Big|\\
   &\leq \Big| \int_\Omega(y^2-x^2)p_{\eta}(x,y)\;d(\eta-\nu)(y)\Big|\\ &+ \Big| \int_\Omega(y^2-x^2)(p_{\eta}(x,y)-p_{\nu}(x,y))\;d\nu(y)\Big|
\end{align*}
Since $p_\eta(x,y)$ is uniformly bounded and Lipschitz in $y$, the integrand in the first term is Lipschitz and so there exists a constant such that 
\begin{align*}
    \operatorname{Lip}_y\bigg((y^2-x^2)p_\eta(x,y)\bigg) \leq C_I.
\end{align*}
By the Kantorovich-Rubinstein duality, 
\begin{align*}
    \bigg|\int_\Omega (y^2 -x^2) p_\eta(x,y) \; d(\eta-\nu)(y)\bigg| &\leq C_I W_1(\eta,\nu).
\end{align*}
For the second term, using $|y^2-x^2|\leq 1$ on $\Omega = [0,1]$ and Lemma \ref{lem: p is Lipschitz W_1},
\begin{align*}
    \bigg|\int_\Omega (y^2-x^2)\big(p_\eta(x,y) -p_\nu(x,y)\big)\, d\nu(y) \bigg| &\leq L_p W_1(\eta,\nu).
\end{align*}
Therefore,
\begin{align*}
    |I_{\eta}(t,x) - I_{\nu}(t,x)|&\leq (C_I + L_p) W_1(\eta,\nu). 
\end{align*}
For the drift term, 
\begin{align*}
    |V_{\eta}(t,x) - V_{\nu}(t,x)| &= \bigg| \frac{I_{\eta}(t,x)}{x+m_{\eta}(t)} - \frac{I_{\nu}(t,x)}{x+m_{\nu}(t)}\bigg| 
    \\
    &\leq \frac{|I_{\eta}(t,x)- I_{\nu}(t,x)|}{x+m_{\eta}(t)} + |I_{\nu}(t,x)|\frac{|m_{\eta}(t)-m_{\nu}(t)|}{(x+m_{\eta}(t))(x+m_{\nu}(t))}\\
    &\leq \frac{|I_{\eta}(t,x) - I_{\nu}(t,x)|}{m_{\eta}(t)} + |I_{\nu}(t,x)|\frac{|m_{\eta}(t)-m_{\nu}(t)|}{m_{\eta}(t)m_{\nu}(t)}\\
    &\leq \frac{|I_{\eta}(t,x)- I_{\nu}(t,x)|}{\delta_T} + \frac{W_1(\eta,\nu)}{{\delta_T}^2}\\
    &\leq \frac{C_I + L_p}{\delta_T} W_1(\eta,\nu) + \frac{1}{{\delta_T}^2}W_1(\eta,\nu)\\
    &= L_1 W_1(\eta,\nu)
\end{align*}
Therefore $V_\eta(t,x)$ is Lipschitz in $\eta$. For continuity in $x$, we begin by showing $I_\eta(t,x)$ is Lipschitz continuous in $x$.
\begin{align*}
    |I_\eta(t,x_1)- I_\eta(t,x_2)| &= \bigg|\int_{\Omega} \bigg((y^2-x_1^2)\,p_\eta(x_1,y)- (y^2-x_2^2)\,\,p_\eta(x_2,y)\bigg)\,d\eta(y)\bigg|\\
    &=\bigg| \int_{\Omega}y^2\Big(p_\eta(x_1,y)-p_\eta(x_2,y)\Big)d\eta(y)\\& + \int_{\Omega} \Big(x_2^2p_\eta(x_2,y) - x_1^2p_\eta(x_1,y)\Big)d\eta(y)\bigg|\\
    &\leq L_p |x_1 - x_2|\;  + \bigg|\int_{\Omega} x_2^2(p_\eta(x_2,y)-p_\eta(x_1,y)) d\eta(y)\bigg|\\ &+ \bigg|\int_{\Omega}(x^2_2-x_1^2)p_\eta(x_1,y)\;d\eta(y)\bigg|\\
    &\leq L_p |x_1 - x_2|\;  +  L_p |x_1 - x_2| + |x^2_2-x_1^2|\\
    &\leq (2+ 2L_p)|x_1-x_2|
\end{align*}
Now consider $V_\eta(t,x)$.
\begin{align*}
    |V_\eta(t,x_1)- V_\eta(t,x_2)| &= \bigg|\frac{I_\eta(t,x_1)}{x_1+m_{\eta}(t)}- \frac{I_\eta(t,x_2)}{x_2+m_{\eta}(t)}\bigg|\\
    &\leq  \frac{\big|I_\eta(t,x_1)-I_\eta(t,x_2)\big|}{x_1+m_{\eta}(t)}+ \frac{|I_\eta(t,x_2)\|x_1-x_2|}{(x_1+m_{\eta}(t))(x_2+m_\eta(t))}\\
    &\leq \frac{1}{\delta_T}(2+ 2L_p)|x_1-x_2| + \frac{1}{(\delta_T)^2}|x_1-x_2|\\ 
    &= \bigg(\frac{1}{\delta_T}(2+ 2L_p) + \frac{1}{(\delta_T)^2}\bigg)|x_1-x_2|\\
    &=L_2 |x_1-x_2|
\end{align*}
Therefore $V_\eta(t,x)$ is Lipschitz in $x$.
\end{proof}

\subsubsection*{Proof of Theorem \ref{theorem: no noise existence and uniqueness}}
\begin{proof}
It suffices to satisfy three regularity conditions given in \cite[Theorem 2]{Bonnet2017ThePM}.
By Lemma \ref{lem: V is Lipschitz W_1}, $V_\mu(t,x)$ is uniformly Lipschitz in $x$.
Moreover, $V_\mu$ is uniformly Lipschitz in $\mu$, so
\begin{align*}
    \|V_{\mu_1}(t,\cdot) - V_{\mu_2}(t,\cdot)\|_{C^0(\Omega)} &\leq L_V W_1(\mu_1,\mu_2)
\end{align*}
where $W_1$ is the Wasserstein metric in 1D \cite{Villani2009}. It suffices to show a linear growth bound on $V_\mu$. The velocity is uniformly bounded, 
    \begin{align*}
        \|V_\mu\|_{C^0(\Omega)} &= \sup_{x \in \Omega}\bigg|\frac{I_\mu(t,x)}{x+m_\mu(t)}\bigg| \leq \frac{1}{\delta_T}\sup_{x \in \Omega}\big|I_\mu(t,x)\big| \leq \frac{1}{\delta_T},
    \end{align*}
and combined with the domain $\Omega = [0,1]$ gives the linear growth bound 
\begin{align*}
    |V_\mu(t,x)| \leq \frac{1}{\delta_T}(1+x).
\end{align*}
These conditions are sufficient to show existence and uniqueness of a solution $\mu(t) \in C^0([0,T], \mathcal{P}_c(\mathbb{R}))$ \cite{Bonnet2017ThePM}.
To show this reduces to the compact interval $\Omega=[0,1]$, we take the velocity at each end point
\begin{align*}
    V_\mu(t,0)  &= \frac{I_\mu(t,0)}{m_\mu(t)} = \frac{1}{m_\mu(t)}\int_\Omega y^2p_\mu(0,y) d \mu(y) \geq 0\\
    V_\mu(t,1)  &= \frac{I_\mu(t,1)}{1+m_\mu(t)} = \frac{1}{1+m_\mu(t)}\int_\Omega (y^2-1)p_\mu(1,y) d\mu(y) \leq 0
\end{align*}
since $y^2-1\leq 0$ everywhere. The Lipschitz continuity and uniqueness of solution ensure the flow does not exceed the domain $[0,1]$, hence the domain $\Omega$ is forward invariant under flow. Hence, if $\mu_0 \in \mathcal{P}_c(\Omega)$ then $\mu(t) \in C^0([0,T], \mathcal{P}_c(\Omega))$.
\end{proof}

\subsubsection*{Proof of Proposition \ref{prop: no noise solution maximum bound}}
\begin{proof}
    The density along characteristics satisfies
    \begin{align*}
        \frac{d}{dt}\mu(t,X(t,x)) &= - \mu(t,X(t,x)) \partial_x V_\mu(t,X(t,x)),
    \end{align*}
    and therefore
    \begin{align*}
        \mu(t,X(t,x)) = \mu_0\; \exp\Big(-\int_0^t\partial_x V_\mu(s,X(s,x))ds \bigg).
    \end{align*}
    Taking the maximum value, 
    \begin{align*}
        \|\mu(t,\cdot)\|_\infty \leq \|\mu_0\|_\infty \; \exp\Big(\int_0^t \|\partial_x V_\mu(s,\cdot)\|_\infty ds \bigg).
    \end{align*}
    By Lemma \ref{lem: mean bounded below}, the mean is bounded below on the finite time interval. Now consider an upper bound on $\partial_xV_\mu(t,x)$. Note that $I_\mu$ and $V_\mu$ are both differentiable in $x$ when $\mu \in L^\infty(\Omega)$.
    \begin{align*}
        \sup_x |\partial_x V_\mu(t,x)| &= \sup_x|\partial_x \frac{1}{x+m_\mu}I_\mu(t,x)|\\
        &= \sup_x|-\frac{I(x,\mu)}{(x+m_\mu)^2} + \frac{1}{x+m_\mu}\partial_xI(x,\mu)|\\
        &\leq |\frac{1}{m_\mu^2} + \frac{1}{m_\mu}\partial_xI_\mu(t,x)|\\
        &\leq \frac{1}{m_\mu^2} + \frac{1}{m_\mu}|\partial_x\int(y^2-x^2)\mu(t,y)p_\mu(x,y)dy|\\
        & = \frac{1}{m_\mu^2} + \frac{1}{m_\mu}|\int(-2x)\mu(t,y)p_\mu(x,y) + (y^2-x^2)\mu(t,y)\partial_xp_\mu(x,y)\; dy|\\
        &\leq \frac{1}{m_\mu^2} + \frac{1}{m_\mu}|2 + \int_\Omega (y^2-x^2)\mu(t,y)\Big(-\beta \int(bz-c)\mu(t,z) dz\; p_\mu(x,y)\\&\qquad \times (1-p_\mu(x,y)\Big)\; dy|\\
        &\leq \frac{1}{m_\mu^2} + \frac{1}{m_\mu}|2 -\frac{\beta}{4} \int(bz-c)\mu(t,z)dz|\\
        &\leq \frac{1}{m_\mu^2} + \frac{2 + \frac{\beta}{4} |\int(bz-c)\mu(t,z)dz|}{m_\mu}\\
        &\leq \frac{1}{m_\mu^2} + \frac{2 + \frac{\beta}{4} b m_\mu + c}{m_\mu}\\
        &\leq \frac{1}{\delta_T^2} + \frac{2+c}{\delta_T} + \frac{\beta b}{4}.
    \end{align*}
    The result then follows substituting the upper bound into the growth bound
    \begin{align*}
        \|\mu(t,\cdot)\|_\infty \leq \|\mu_0\|_\infty \; \exp\bigg[t\big(\frac{1}{\delta_T^2} + \frac{2+c}{\delta_T} + \frac{\beta b}{4}\big)\bigg].
    \end{align*}
\end{proof}

\subsubsection*{Proof of Theorem \ref{thm: convergence of measure to a dirac}}
\begin{proof}
    Let $l(t) = \inf \text{supp}(\mu_t)$ and $r(t) = \sup \text{supp}(\mu_t)$ denote the left and right limits of the support, with the distance defined as $D(t) = r(t)-l(t)$. Note that if $r=l$ then $\mu_t = \delta_l$ and the proof is complete. In particular, if $\mu_0 = \delta_0$ the result is immediate, hence we may assume $m_{\mu_0} >0$. We will show that the selection term is bounded below, and use this to show bound the velocity at each end point. The mean is bounded above by $1$, hence for $x\in[0,1]$, 
    \begin{align*}
        \bigg|\beta(y-x)[(b-c)m_\mu(t) - c(y+x)]\bigg|&\leq \beta|(b-c)m_\mu(t) | + 2\beta c\leq \beta(b+3c) \,.
    \end{align*}
    Since the logistic function is increasing in its argument, $p_\mu$ is minimised at $-\beta(b+3c)$:
    \begin{align*}
        p_\mu(x,y) \geq \frac{1}{1+\exp(\beta(b+3c))} =: p_{\min} >0
    \end{align*}
    Then the non-local integral at the boundary of the support is bounded by
    \begin{align*}
        I(l) &= \int_\Omega (y^2-l^2) p_\mu(l,y)d\mu(y)\\
        &\geq p_{\min} \int_\Omega(y^2-l^2)d\mu(y) \\
        &\geq0\\
        I(r) &= \int_\Omega (y^2-r^2) p_\mu(r,y)d\mu(y)\\
        &\leq p_{\min} \int_\Omega(y^2-r^2)d\mu(y)\\
        &\leq 0
    \end{align*}
    with equality only when $y=l$ a.e. or $y=r$ a.e. respectively, hence when $D=0$. Since the end points evolve along the characteristics and $x +m_\mu(t) >0$ for all $x \in [0,1]$, we have
    \begin{align*}
        \frac{d}{dt}l = \frac{I_\mu(t,l)}{l + m_\mu(t)} \geq0,\qquad
        \frac{d}{dt}r = \frac{I_\mu(t,r)}{r + m_\mu(t)} \leq0
    \end{align*}
    with equality only when $D=0$. Consider the time derivative of the distance between the support bounds,
    \begin{align*}
        \dot{D}(t) &= \dot{r}(t) - \dot{l}(t)\\
        &=  \frac{I_\mu(t,r)}{r + m_\mu(t)} -  \frac{I_\mu(t,l)}{l + m_\mu(t)}\\
        &\leq \frac{p_{\min}}{r+m_\mu(t)} \int_\Omega(y^2-r^2) d\mu(y) -  \frac{p_{\min}}{l+m_\mu(t)} \int_\Omega(y^2-l^2)d\mu(y)\\ 
        \intertext{Since $l+m_\mu(t) \leq r+m_\mu(t)$,}
        &\leq \frac{-p_{\min}}{r+m_\mu(t)}  \int_\Omega (r^2-l^2)d\mu(y)\\
         &= \frac{-p_{\min}(r^2-l^2)}{r+m_\mu(t)}\\
         &= \frac{-p_{\min}(r+l)}{r+m_\mu(t)}D(t)
    \end{align*}
    We can then use the fact for all $r>0$, $\frac{r+l}{r+m_\mu(t)}  \geq \frac{r}{2r}=\frac{1}{2}$, and apply Gr\"onwall's Lemma to get 
    \begin{align*}
        D(t) &\leq D(0) e^{-\frac{p_{\min}}{2} t} \to 0
    \end{align*}
    Since $\mu(t)$ is a probability measure supported on $[l(t),r(t)]$, and the intervals collapses to zero, there exists an $x^* \in [0,1]$ such that $\mu(t,\cdot)\rightarrow\delta_{x^*}$ weakly in $\mathcal{P}_c(\Omega)$.
\end{proof}

\subsubsection*{Proof of Proposition \ref{prop: unique solution for complete network}}
\begin{proof}
    Since the interaction function is still Lipschitz, existence and uniqueness of a solution follows from \cite{Bonnet2017ThePM}. Convergence to a Dirac follows from a similar argument to that in Theorem \ref{thm: convergence of measure to a dirac}.
\end{proof}

\subsubsection*{Proof of Lemma \ref{lem: S_mu forward invariant}}
\begin{proof}
First note that if $ S \geq 0$ then $p_\eta(z,y) \geq \frac{1}{2}$. This condition is 
\begin{align*}
    (b-c)m_\mu(t) - c(y+z) &\geq0\\
    \Big((b-c)w_z-c\Big)z + \Big((b-c)w_y-c\Big)y &\geq 0\\
    \alpha_1 z + \alpha_2y &\geq 0
\end{align*}
where $\alpha_1: = (b-c)w_z-c$ and $\alpha_2: = (b-c)w_y-c$. The define the region 
\begin{align*}
    X = \{(z,y) : 0 \leq z \leq y \leq 1, \alpha_1 z + \alpha_2y &\geq 0 \}
\end{align*}
Then the time derivative of the boundary line is
\begin{align*}
    \frac{dS}{dt} &= \alpha_1 \frac{dz}{dt}+\alpha_2\frac{dy}{dt}\\
    &= (y^2-z^2)\Big(\frac{\alpha_1w_y p_\mu(z,y)}{z+m_\mu} - \frac{\alpha_2w_z (1-p_\mu(z,y))}{y+m_\mu}\Big)\\
    \intertext{Now substitute $p_\mu(z,y) = \frac{1}{2}$ and $\alpha_2 = -\alpha_1\frac{z}{y}$,}
    &=\alpha_1\frac{(y^2-z^2)}{2}\Big(\frac{w_y}{z+m_\mu} + \frac{w_zz}{y(y+m_\mu)}\Big)
\end{align*}
Therefore, along the boundary, the sign of the derivative is the sign of $\alpha_1$. Thus the sufficient condition for forward invariance is $(b-c)w_z-c\geq0$. It remains to show that the region is question is non-trivial, and contains the initial condition $(z,y) = (0,1)$, which requires $\alpha_2 \geq 0$. For positive gradient along the boundary, we have $\alpha_1 \ge 0$, hence the area in question is $\alpha_1,\alpha_2 \ge 0$. These conditions reduce to the inequality on $w_y$,
\begin{align*}
    \frac{c}{b-c} \leq w_y \leq 1- \frac{c}{b-c}
\end{align*}
which exists for all $b\geq3c$. 
\end{proof}

\subsubsection*{Proof of Proposition \ref{proposition: cooperation is non-decreasing}}
\begin{proof}
    The overall cooperation level is given by mean. The time derivative is this given by
    \begin{align*}
        \frac{d}{dt}m_\mu(t) &= (y^2-z^2)w_zw_y\bigg(\frac{p_\mu(z,y)}{D_1} - \frac{1-p_\mu(z,y)}{D_2}\bigg) \,,
    \end{align*}
    where $D_1:=z+m_\mu(t), D_2:=y+m_\mu(t)$ are the denominators of \eqref{eqn: two type ode dz/dt} and \eqref{eqn: two type ode dy/dt}, respectively.
    Since $y \geq z$, we have $D_2 \geq D_1> 0$, and therefore the quotient $\frac{D_2}{D_1} \geq 1$. Factorising the above expression, a sufficient condition for non-negativity of the time derivative reduces to showing
    \begin{align*}
        \frac{p_\mu(z,y)}{D_1} - \frac{1-p_\mu(z,y)}{D_2} &\geq 0\\
        \Leftrightarrow\ \frac{p_\mu(z,y)}{1-p_\mu(z,y)} &\geq \frac{D_1}{D_2},
    \end{align*}
    which is true if $p_\mu(z,y) \geq 1/2, \forall t\geq 0$, which holds by Lemma \ref{lem: S_mu forward invariant} under Assumption \ref{reward and weight assumption}. 
\end{proof}

\subsubsection*{Proof of Proposition \ref{prop: limit point as function of beta}}

\begin{proof}
    We will prove the properties in order. Let $\beta=0$, then the switching probability $p_\mu(z,y)=\frac{1}{2}, \forall t\geq 0$. The coupled ODEs reduce to
    \begin{align*}
        \frac{dz}{dt} &= \frac{1}{2}\frac{w_y(y^2 - z^2)}{z+m_\mu} \\[0.2em]
        \frac{dy}{dt} &= -\frac{1}{2}\frac{w_z(y^2 - z^2)}{y+m_\mu} 
    \end{align*}
    The non-parametric form is given by
    \begin{align*}
        \frac{dy}{dz} &= -\frac{w_zD_1}{w_yD_2}\\
        \Leftrightarrow D_2\;dy + kD_1\; dz &= 0
    \end{align*}
    where $D_1,D_2$ are the denominators of $\dot{z},\dot{y}$ respectively, and $k = \frac{w_z}{w_y}$. 
    We show that the potential function $H(z,y)$,    
    \begin{align*}
        H(z,y) &:= \frac{k(1+w_z)z^2}{2} + w_zzy + \frac{1+w_y}{2}y^2
    \end{align*}
    satisfies the equality
     \begin{align*}
        dH = kD_1\; dz + D_2\;dy.
    \end{align*}
    The derivative of $H$ can be expressed as  
    \begin{align*}
        dH(z,y) &= \frac{\partial H}{\partial z} dz + \frac{\partial H}{\partial y} dy.
    \end{align*}
    Then note that 
    \begin{align*}
        \frac{\partial H}{\partial z} &= k(1+w_z)z+w_zy\\ &= k(z+w_zz + w_yy)\\
        &= k(z+m_\mu)\\
        & = kD_1\\
        \frac{\partial H}{\partial y} &= w_zz+(1+w_y)y\\ &= y+w_zz + w_yy)\\
        &= y+m_\mu\\
        & = D_2
    \end{align*}
    Therefore this potential function satisfies the equation. Substituting the start and end point:
    \begin{align*}
        H(0,1) &= \frac{1+w_y}{2},\qquad H(L,L) = \frac{L^2}{w_y}
    \end{align*}
    Equating the two, and rearranging for $L$ gives
    \begin{align*}
        L(0) = \sqrt{\frac{w_y(1+w_y)}{2}}.
    \end{align*}
    To prove the second property, we note that for the strict bounds in Assumption \ref{reward and weight assumption}, the exponential term $(b-c)m_\mu(t)-c(y+z) >0$, and therefore 
    \begin{align*}
        \lim_{\beta \rightarrow\infty} p_\mu(z,y) = 1.
    \end{align*}
    In this case, the system of ODEs reduces to 
    \begin{align*}
            \frac{dz}{dt} &= \frac{w_y(1 - z^2)}{z+m_\mu},\\
            \frac{dy}{dt} &= 0,
    \end{align*}
    which converges to $(z,y) = (1,1)$, and hence $L=1$. For monotonicity, consider the non-parametric dynamics given by
    \begin{align*}
        f(z,y,\beta):=\frac{dy}{dz} = -\frac{w_zD_1}{w_yD_2} \frac{1-p_\mu(z,y)}{p_\mu(z,y)}
    \end{align*}
    where $\frac{w_zD_1}{w_yD_2} >0$ on the domain $\Omega\setminus(0,0)$. Under Assumption \ref{reward and weight assumption} and with $y >z$, then $p_\mu$ is strictly increasing in $\beta$. Taking the derivative of $f$,
    \begin{align*}
        \frac{df}{dp_\mu(z,y)}  = \frac{w_zD_1}{w_yD_2} \frac{1}{p_\mu(z,y)^2} > 0, 
    \end{align*}
    implies that $f$ is strictly increasing in $p_\mu$ and therefore in $\beta$. This prevents the intersection of two trajectories of $y(z)$ with different $\beta$s. This means for $\beta_2 > \beta_1$,
    \begin{align*}
        y_{\beta_2}(z) > y_{\beta_1}(z),\;\; \forall z \in (z_0, \min\{L(\beta_1),L(\beta_2)\})
    \end{align*}
    Therefore, the limiting points are such that $L(\beta_2) \geq L(\beta_1)$.
\end{proof}

\section{Proofs: Section \ref{Section: With noise}} \label{Appendix: With noise}

\subsubsection*{Proof of Proposition \ref{prop: pde properties}}
\begin{proof}
    Let $(t^*,x^*)$ be a point such that $\mu(t^*,x^*) = 0$. Since the solution is continuous, if this point doesn't exist then $\mu(t,x) \geq 0$ for all $t$. Since $\mu(t,\cdot) \ge 0$ for $t \leq t^*$, the point $(t^*,x^*)$ is a local minimum of $\mu(t^*,\cdot)$. As such, the first and second derivatives are zero and non-negative respectively. 
    Equation \eqref{eqn: mean-field pde} evaluated at $(t^*,x^*)$ reduces to
    \begin{align*}
        \partial_t\mu(t,x)\bigg|_{(t^*,x^*)} &= \frac{\sigma^2}{2}\partial_{xx}\mu|_{(t^*,x^*)},\\
        &\geq 0.
    \end{align*}
Therefore, the solution remains non-negative for all $t$. For mass conservation, 
    \begin{align*}
        \partial_t \bigg( \int_\Omega \mu(t,x) \, dx \bigg) 
        &= \int_\Omega \partial_t \mu(t,x) \, dx \\
        &= -\int_\Omega \partial_xF[\mu](t,x) \, dx \\
        &= F[\mu](t,0) - F[\mu](t,1) \\
        &= 0 
    \end{align*}
    due to the no-flux boundary conditions. 
\end{proof}

\subsubsection*{Proof of Lemma \ref{lemma: steady state, bounded solution}}
\begin{proof}
We begin by showing $|b_\eta(x)|$ is bounded.
\begin{align*}
    |b_\eta(x)| &= \bigg| \frac{1}{\omega}\int_{\Omega}(y^2-x^2)\eta(y)p_{\eta}(x,y)\, dy \,\bigg| \\
    &\leq \frac{1}{\omega}\int_{\Omega} \Big| (y^2-x^2)\eta(y)p_{\eta}(x,y)\Big|\, dy \, \\
    &= \frac{1}{\omega}\int_{\Omega} \Big| (y^2-x^2)\eta(y)\Big|\, dy \, \\
    &= \frac{1}{\omega}\int_{\Omega} | y^2-x^2 |\,\eta(y)\, dy \, \\
    &\leq \frac{1}{\omega}\int_{\Omega} \,\eta(y)\, dy \, \\
    &\leq \frac{1}{\omega}\,. 
\end{align*}
Then,  Equation \eqref{eqn: stationary ode} is bounded by
\begin{align}
    |\frac{d}{dx}\mu(x)| &\leq \frac{\mu(x)}{x+m_\eta}\frac{1}{\omega}.
\end{align}
This has solution bounds
\begin{align*}
    \mu_0 \bigg(1+\frac{x}{m_\eta}\bigg)^{-\frac{1}{\omega}} \leq \mu(x) \leq \mu_0\bigg(1+\frac{x}{m_\eta}\bigg)^{\frac{1}{\omega}}\\ 
    \mu_0 \bigg(1+\frac{1}{m_\eta}\bigg)^{-\frac{1}{\omega}} \leq \mu(x) \leq \mu_0\bigg(1+\frac{1}{m_\eta}\bigg)^{\frac{1}{\omega}}
\end{align*}
where the final line follows from a uniform bound over $x\in[0,1]$.
\end{proof}

\subsubsection*{Proof of Lemma \ref{lem: m^* unique}}

\begin{proof}
    A solution to $g(m,\omega)$ exists if and only if
    \begin{align*}
        2m &= (1+\frac{1}{m})^{-\frac{2}{\omega}}= (\frac{m}{1+m})^{\frac{2}{\omega}}
    \end{align*}
    Let $z = \frac{m}{1+m} \in (0,1)$, then 
    \begin{align*}
        2\frac{z}{1-z} &= z^{\frac{2}{\omega}}\\
        2 &= z^{\frac{2}{\omega}-1}(1-z)=: f(z)
    \end{align*}
    Then if $\frac{2}{\omega}-1 \geq 0$, we have $z^{\frac{2}{\omega}-1}(1-z) \leq 1$ and therefore there are no solutions. Now consider $\frac{2}{\omega} -1 < 0 \Leftrightarrow \omega > 2$, then taking the limits in $z$
    \begin{align*}
        \lim_{z\rightarrow0^+} f(z) &= +\infty\\
        \lim_{z\rightarrow1^-} f(z) &= 0
    \end{align*}
    Since $f(z)$ is a continuous function, by the intermediate value theorem there exists a $z^* \in (0,1)$ such that $f(z^*) = 2$. Then taking $m^* = \frac{z^*}{1-z^*} \in (0,\infty)$ gives the solution for $m$. Now take the derivative of $f(z)$, 
    \begin{align*}
        f'(z) &=z^{\frac{2}{\omega}-2}(\frac{2}{\omega}-1 - \frac{2}{\omega}z)
    \end{align*}
    The front term is always positive, and the bracket term
    \begin{align*}
        \frac{2}{\omega}-1 - \frac{2}{\omega}z <  \frac{2}{\omega}-1  < 0
    \end{align*}
    for $\omega>2$. Therefore the derivative of $f(z)$ is always negative, and the solution $z^*$ is unique. To reduce the bounds on $m^*$, note that $g(\frac{1}{2},\omega) >0$ and $\lim_{m\rightarrow0^+} g(m,\omega) \approx m - \frac{1}{2}m^{\frac{2}{\omega}}<0$ for sufficiently small $m$. Therefore, the unique solution, $m^*$, must lie in the range $(0,\frac{1}{2})$. This also guarantees that if $m_\eta > m^*$, then $g(m_\eta,\omega)>0$, since there is a unique point at which $g(m,\omega)=0$.
\end{proof}

\subsubsection*{Proof of Lemma \ref{lem: b is Lipschitz L_infty}}
\begin{proof}
    For any $x\in \Omega$,
\begin{align*}
    |b_{\eta}(x) - b_{\nu}(x)|&= \Big|\frac{1}{\omega} \int_\Omega(y^2-x^2)\bigg(\eta p_{\eta}(x,y)-\nu p_{\nu}(x,y)\bigg)dy\Big|\\
    &= \frac{1}{\omega}\Big|\int_\Omega(y^2-x^2)\bigg((\eta-\nu)p_{\eta}(x,y)+\nu(p_{\eta}(x,y)-p_{\nu}(x,y))\bigg)dy\Big|\\
    &\leq  \frac{1}{\omega}\int_\Omega|y^2-x^2|\Big(|\eta-\nu|p_{\eta}(x,y)+|\nu\|p_{\eta}(x,y)-p_{\nu}(x,y)|\bigg)dy\\
    &\leq \frac{1}{\omega}\Big(\|\eta-\nu\|_\infty +L_p \|\eta-\nu\|_\infty\Big)
    \intertext{where we have used $\|\nu\|_{L^1} = 1$ and extended Lemma \ref{lem: p is Lipschitz W_1} to the $L^\infty$ norm. Taking the supremum over $x \in \Omega$}
    \|b_{\eta} - b_{\nu}\|_\infty &\leq \frac{1}{\omega}(1+L_p) \|\eta-\nu\|_\infty.
\end{align*}
 Therefore $b_\eta$ is Lipschitz continuous in $\eta$ with respect to the $L^\infty$ norm.
\end{proof}

\subsubsection*{Proof of Lemma \ref{lem: F well defined and invariant}}
\begin{proof}
    Let $\eta \in K$. As $\eta \in L^\infty(\Omega)$, then $b_\eta$ is continuous by Lemma \ref{lem: b is Lipschitz L_infty}. Therefore $\mathcal{F}[\eta]$ as defined in \eqref{eqn: solution to stationary distribution ODE} is the unique strong solution of \eqref{eqn: stationary ode}. Moreover $\mathcal{F}[\eta]$ is continuous and therefore also an element of $L^\infty(\Omega)$. 
    
    Since $m_\eta > m^*>0$, the exponential guarantees positivity of $\mathcal{F}[\eta]$. 
    
    Clearly,
    \begin{align*}
        \int_{\Omega}\mathcal{F}[\eta](x)dx &= \int_{\Omega}\frac{w_\eta(x)}{\int_{\Omega}w_v(s)ds}dx\\
        &=\frac{1}{\int_{\Omega}w_\eta(s)ds}\int_{\Omega} w_\eta(x)dx\\
        &=1 \,.
    \end{align*}
    
    It remains to show that $m_{\mathcal{F}[\eta]} \geq m^*$. Recall
    \begin{align*}
        m_{\mathcal{F}[\eta]} &\geq \frac{1}{2}\bigg(1+\frac{1}{m_\eta}\bigg)^{-\frac{2}{\omega}}\\
        &\geq \frac{1}{2}\bigg(1+\frac{1}{m^*}\bigg)^{-\frac{2}{\omega}}\\
        &\geq m^* \,.
    \end{align*}
    Therefore $\mathcal{F}[\eta] \in K$, and thus $\mathcal{F}(K) \subset K$.
\end{proof}

\subsubsection*{Proof of Lemma \ref{lem: F(K) is compact}}
\begin{proof}
    From the Arzela-Ascoli Theorem \cite{rudin_principles}, it suffices to show that the set $\mathcal{F}(K)$ is equibounded and equicontinuous. Firstly, by Proposition \ref{prop: F_bounded}, the set is uniformly bounded by $M_\mathcal{F}$. Furthermore, 
    \begin{align*}
        |\mathcal{F}'[\eta](x)| &= |\frac{b_\eta(x)}{x+m_\eta}\mathcal{F}[\eta](x)|\\
        &\leq \frac{2}{\omega m^*}M_\mathcal{F}.
    \end{align*}
    This gives equicontinuity uniformly in $\eta$, and hence $\mathcal{F}(K)$ is relatively compact. The equivalence of the norms in $C^\infty(\Omega)$ and $L^\infty(\Omega)$, implies relative compactness in $L^\infty(\Omega)$.
\end{proof}

\subsubsection*{Proof of Proposition \ref{prop: F_bounded}}
\begin{proof}
    Recall that $|b_\eta(x)| \leq \frac{1}{\omega}$, then it directly follows that
    \begin{align*}
        \Big|\int_0^x\frac{b_\eta(z)}{z+m}dz\Big|\leq \frac{1}{\omega m^*}x.
    \end{align*}
    \begin{align*}
        e^{-\frac{1}{\omega m^*}x} \leq w_\eta(x) \leq e^{\frac{1}{\omega m^*}x}
    \end{align*}
    This means
    \begin{align*}
        \int_{\Omega}w_\eta(s) ds \geq \int_{\Omega}e^{-\frac{1}{\omega m^*}s}ds &= \frac{\omega m^*(1-e^{-\frac{1}{\omega m^*}})}{2}\\
        \sup_{x\in[0,1]}w_\eta(x) &\leq e^\frac{1}{\omega m^*}.
    \end{align*}
    Therefore,
    \begin{align*}
        \|\mathcal{F}[\eta]\|_{\infty} &= \sup_{x \in[0,1]}\frac{w_\eta(x)}{\int_{\Omega}w_\eta(s)ds}\\
        &\leq \frac{2e^\frac{1}{\omega m^*}}{{\omega m^*}(1-e^{-\frac{1}{\omega m^*}})} :=M_\mathcal{F}
    \end{align*}
\end{proof}

\subsubsection*{Proof of Proposition \ref{prop: F operator is continuous}}
\begin{proof}
    Consider two functions $\eta, \nu \in K$. Let $Y_\eta = \int_{\Omega}w_{\eta}(s)ds$, where $|Y_\eta| \geq \frac{\omega m^*}{1}(1-e^{-\frac{1}{\omega m^*}}):= \frac{1}{C_Y}$ by Proposition \ref{prop: F_bounded}. Then, for any $x \in [0,1]$
    \begin{align*}
        |\mathcal{F}[\eta](x) - \mathcal{F}[\nu](x)| &= |\frac{w_{\eta}}{Y_\eta}-\frac{w_{\nu}}{Y_\nu}|\\
        &\leq \frac{|w_{\eta}-w_{\nu}|}{|Y_\eta|} + \frac{|w_{\eta}\|Y_\eta-Y_\nu|}{|Y_\eta Y_\nu|}\\
        &\leq C_Y|w_{\eta}-w_{\nu}| + e^{\frac{1}{\omega m^*}} C_Y^2|Y_\eta-Y_\nu|\\
        \intertext{Thus, taking the supremum over $\Omega$}
        \|\mathcal{F}[\eta](x) - \mathcal{F}[\nu](x)\|_\infty &\leq C_Y\|w_{\eta}-w_{\nu}\|_\infty + e^{\frac{1}{\omega m^*}}C_Y^2|Y_\eta-Y_\nu|
    \end{align*}
It remains to bound the terms $\|w_\eta-w_\nu\|_\infty$ and $|Y_\eta -Y_\nu|$. Taking them in turn,
\begin{align*}
    |w_{\eta}-w_{\nu}| &= |e^{V_\eta}-e^{V_\nu}|\\
    &\leq e^{\max{\{V_\eta,V_\nu\}}}|V_\eta -V_\nu| \qquad\text{by the mean-value theorem,}\\
    &\leq e^{\frac{1}{\omega m^*}}|V_\eta -V_\nu| \\
    &= e^{\frac{1}{\omega m^*}} \Big|\int_{\Omega}\frac{b_\eta}{s+m_{\eta}(t)} - \frac{b_\eta}{s+m_{\eta}(t)}ds\Big|\\
    &\leq e^{\frac{1}{\omega m^*}} \int_{\Omega}|\frac{b_\eta}{s+m_{\eta}(t)} - \frac{b_\eta}{s+m_{\eta}(t)}|ds\\
    &\leq e^{\frac{1}{\omega m^*}}\int_\Omega \frac{|b_\eta-b_\nu|}{|s+m_{\eta}(t)|} + \frac{|b_\nu|\;|m_{\eta}(t)-m_{\nu}(t)|}{|(s+m_{\eta})(s+m_{\nu}(t))|}ds\\
    &\leq e^{\frac{1}{\omega m^*}} \Big(\frac{|b_\eta-b_\nu|}{m^*} +\frac{2|m_{\eta}(t)-m_{\nu}(t)|}{{\omega m^*}^2}\Big)
\end{align*}
Note that Lemma \ref{lem: p is Lipschitz W_1} extends to the $L^\infty$ norm for densities (proof is the same). For any $x$,
\begin{align*}
    |b_\eta(x)-b_\nu(x)|&= \Big|\frac{1}{\omega} \int_\Omega(y^2-x^2)\bigg(\eta p_{\eta}(x,y)-\nu p_{\nu}(x,y)\bigg)dy\Big|\\
    &= \Big|\frac{1}{\omega} \int_\Omega(y^2-x^2)\bigg((\eta-\nu)p_{\eta}(x,y)+\nu(p_{\eta}(x,y)-p_{\nu}(x,y))\bigg)dy\Big|\\
    &\leq \frac{1}{\omega} \int_\Omega|y^2-x^2|\Big(|\eta-\nu|p_{\eta}(x,y)+|\nu\|p_{\eta}(x,y)-p_{\nu}(x,y)|\bigg)dy\\
    &\leq \frac{2}{3\omega}\Big(|\eta-\nu| + M_{\mathcal{F}}L_p \|\eta-\nu\|_\infty\Big)\intertext{Thus, taking the supremum over $x$}
    \|b_\eta-b_\nu\|_\infty &\leq \frac{2}{3\omega}(1+M_{\mathcal{F}}L_p ) \|\eta-\nu\|_\infty
\end{align*}
Finally,
\begin{align}
    |Y_\eta-Y_\nu| &= \Big|\int_\Omega w_\eta(x) - w_\nu(x)dx\Big|\\
    &\leq \|w_\eta(x) - w_\nu(x)\|_\infty
\end{align}
Bringing the terms together, we have that
\begin{align*}
    \|\mathcal{F}[\eta]-\mathcal{F}[\nu]\|_\infty &\leq (C_Y + e^{\frac{1}{\omega m^*}}C_Y^2)\|w_\eta-w_\nu\|_\infty\\
    &\leq (C_Y + e^{\frac{1}{\omega m^*}}C_Y^2)e^{\frac{1}{\omega m^*}}\Big(\frac{2}{3\omega m^*}\Big(1+M_\mathcal{F}L_p \Big) + \frac{1}{\omega{m^*}^2}\Big)\|\eta-\nu\|_\infty \\
    &\leq \frac{1}{\omega m^*}M_\mathcal{F}(1+M_\mathcal{F})\Big(1 + M_\mathcal{F}L_p + \frac{1}{m^*}\Big)\|\eta-\nu\|_\infty
\end{align*}
\end{proof}

\subsubsection*{Proof of Lemma \ref{lem: L_F is monotone decreasing in m}}
\begin{proof}
We begin by solving Equation \eqref{eqn: g(m,omega)} and rearranging to get $1/\omega$:
    \begin{align*}
        \Upsilon(m^*) :=\frac{1}{\omega}&= -\frac{1}{2}\frac{\ln(2m^*)}{\ln(1+\frac{1}{m^*})}.
    \end{align*}
For simplicity in taking derivatives, define  
    \begin{align*}
        \Psi(m^*)&= \frac{\Upsilon(m^*)}{m^*} = -\frac{1}{2m^*}\frac{\ln(2m^*)}{\ln(1+\frac{1}{m^*})}\\
        M_\mathcal{F}(m^*) &= \frac{\Psi e^\Psi}{1-e^{-\Psi}}
    \end{align*}
which reduces the Lipschitz constant to
\begin{align*}
    L_\mathcal{F}(m^*) :=\Psi M_\mathcal{F}(1+M_\mathcal{F})\Big(1 + M_\mathcal{F}L_p + \frac{1}{m^*}\Big).
\end{align*}
We will show that each of $\Upsilon, \Psi, M_\mathcal{F}$ are monotone decreasing in $m^*$.
\begin{align}
    \frac{d\Upsilon(m^*)}{dm^*} &= -\frac{1}{2}\frac{\frac{1}{m^*}\ln(1+\frac{1}{m^*}) +\frac{1}{m^*(m^*+1)}\ln(2m^*)}{(\ln(1+\frac{1}{m^*}))^2}
\end{align}
The numerator can be bounded below for all $m^*>0$
\begin{align*}
    \frac{1}{m^*}\ln(1+\frac{1}{m^*}) +\frac{1}{m^*(m^*+1)}\ln(2m^*) &> \frac{1}{m^*(m^*+1)}\ln(1+\frac{1}{m^*})\\ &+\frac{1}{m^*(m^*+1)}\ln(2m^*)\\
    &= \frac{1}{m^*(m^*+1)}\Big[\ln(2m^*+1)\Big]\\
    &>0.
\end{align*}
Since the denominator is always positive,
\begin{align}
    \frac{d\Upsilon(m^*)}{dm^*} &< 0 \,.
\end{align}
Now considering the derivative of $\Psi(m^*)$:
\begin{align*}
    \frac{d\Psi}{dm^*} &= \frac{1}{(m^*)^2}(m^*\Upsilon' -\Upsilon)
    \leq -\frac{\Upsilon(m^*)}{m^*} < 0
\end{align*}
since $\Upsilon(m^*)>0$ for all $m^* \in (0,1/2)$. Finally,
\begin{align*}
    \frac{dM_\mathcal{F}}{d\Psi} &= \frac{e^{\Psi}(1+\Psi) - (1+2\Psi)}{(1-e^{-\Psi})^2}\\
    \intertext{The numerator has positive derivative in $\Psi$ and is $0$ when $\Psi=0$, hence is positive for $\Psi>0$. Thus }
    \frac{dM_\mathcal{F}}{d\Psi} & > 0\\
    \Rightarrow \frac{dM_\mathcal{F}}{dm^*} &= \frac{dM_\mathcal{F}}{d\Psi}\frac{d\Psi}{dm^*} <0
\end{align*}
We now want to show that $\frac{dL_\mathcal{F}}{dm} < 0$:
\begin{align*}
    L_\mathcal{F}' &= (\Psi M_\mathcal{F}(1+M_\mathcal{F}))'(1 + M_\mathcal{F}L_p+\frac{1}{m^*}) + (\Psi M_\mathcal{F}(1+M_\mathcal{F}))(M_\mathcal{F}'L_p-\frac{1}{(m^*)^2})
\end{align*}
Note that $\Psi$ and $M_\mathcal{F}$ are positive on the domain. The two derivative terms are bounded above by zero:
\begin{align*}
    \Big(\ln(\Psi M_\mathcal{F}(1+M_\mathcal{F}))\Big)' &= \frac{\Psi'}{\Psi} + \frac{M_\mathcal{F}'}{M_\mathcal{F}} + \frac{M_\mathcal{F}'}{1+M_\mathcal{F}} <0 \Rightarrow \Big(\Psi M_\mathcal{F}(1+M_\mathcal{F})\Big)' <0\\
    (M_\mathcal{F}'L_p-\frac{1}{(m^*)^2}) &\leq -\frac{1}{(m^*)^2} <0
\end{align*}
which means that $L_\mathcal{F}'$ is strictly negative.
\end{proof}

\subsubsection*{Proof of Lemma \ref{lem: m(omega) is monotone increasing in omega}}
\begin{proof}
Taking partial derivatives,
    \begin{align*}
        \frac{\partial g}{\partial m} &= 1-2\frac{(1+\frac{1}{m})^{-2/\omega}}{\omega m(m+1)}\\
        \frac{\partial g}{\partial m}\Bigg|_{g=0} &= 1-\frac{2}{\omega(m+1)}\\
        &>0\\
    \\
        \frac{\partial g}{\partial \omega} &=-\frac{1}{2}(1+\frac{1}{m})^{-2/\omega}(\frac{2}{\omega^2}\ln(1+\frac{1}{m}))\\
         \frac{\partial g}{\partial \omega}\Bigg|_{g=0}&= -2\frac{m}{\omega^2}\ln(1+\frac{1}{m})\\
        &<0 
    \end{align*}
By the Implicit Function Theorem, $m(\omega)$ is differentiable and 
\begin{align*}
    \frac{dm}{d\omega} &= -\frac{g_\omega}{g_m}\Big|_{g=0}>0
\end{align*}
Therefore $m(\omega)$ is strictly increasing in $\omega$ on $(2,\infty)$.
\end{proof}
The critical value of $m$ will be such that $L_\mathcal{F}$ is exactly 1. The values of $m$ and $\omega$ are coupled by the relation \eqref{eqn: g(m,omega)}. Taking the limit in $\omega$:
\begin{align*}
    &\omega\rightarrow 2^+,\; m\rightarrow 0^+,\; L_\mathcal{F}(m) \rightarrow \infty.\\
    &\omega\rightarrow \infty,\; m\rightarrow \frac{1}{2}^-,\; L_\mathcal{F}(m) \rightarrow 0.
\end{align*}

\subsubsection*{Proof of Theorem \ref{thm: unique fixed point}}
\begin{proof}
    By Lemma \ref{lem: m^* unique}, for $\omega>2$ there is a unique $m^*(\omega) \in (0,\frac{1}{2})$. Lemma \ref{lem: m(omega) is monotone increasing in omega} ensures that the mean $m^*(\omega)$ is continuous and strictly increasing in $\omega$. By Lemma \ref{lem: L_F is monotone decreasing in m}, the Lipschitz constant $L_\mathcal{F}(m^*(\omega))$ is continuous and strictly decreasing in $m^*(\omega)$. As such, $L_\mathcal{F}(m^*)$ is continuous and monotone decreasing in $\omega$. Taking the limits in $\omega$,
    \begin{align*}
    &\lim_{\omega\rightarrow 4^+} L_\mathcal{F}(m^*(\omega)) = \infty,\\
    &\lim_{\omega\rightarrow \infty} L_\mathcal{F}(m^*(\omega)) = 0.
\end{align*}
    By the intermediate value theorem, there exists an $\omega^*$ such that $L_\mathcal{F}(m^*(\omega^*))=1$. By monotonicity, for all $\omega > \omega^*, L_\mathcal{F}(m^*(\omega)) < 1$. Therefore the Lipschitz constant is less than one, and $\mathcal{F}$ is a contraction mapping. By Banach's Fixed Point Theorem, there is a unique fixed point of $\mathcal{F}$ in $K$, and therefore a unique solution to the steady-state equation.
\end{proof}

\subsubsection*{Proof of Proposition \ref{prop: uniform is not stationary}}
\begin{proof}
    The uniform distribution on $[0,1]$ is given by $\mu(x)=1$ for all $x\in[0,1]$. Assume for a contradiction that this distribution is stationary.
    
    Note that
    \begin{align*}
        \partial_x p_\mu(x,y)\big|_{x=0} &= -\frac{\beta(b-c)}{2} p_\mu(0,y) (1- p_\mu(0,y)).
    \end{align*}
    
    When $\mu$ is the uniform distribution, Equation \eqref{eqn: steady_state} becomes
    \begin{align*}
        0 &= \partial_x \Bigg( \frac{1}{x+1/2} \int_{\Omega} (y^2-x^2)\,p_\mu(x,y)\,dy\Bigg) \\
        \iff 0 &= \partial_x\Bigg( \frac{1}{2x+1} \int_{\Omega} (y^2-x^2)\,p_\mu(x,y)\,dy\Bigg) \\
        \iff 0 &= -\frac{2}{(2x+1)^2} \int_{\Omega} (y^2-x^2)\,p_\mu(x,y)\,dy + \frac{1}{2x+1} \,\partial_x \bigg( \int_{\Omega} (y^2-x^2)\,p_\mu(x,y)\,dy \bigg)\\
        \iff 0 &= -\frac{2}{(2x+1)^2} \int_{\Omega} (y^2-x^2)\,p_\mu(x,y)\,dy 
        \\&+ \frac{1}{2x+1} \int_{\Omega} \Big( -2x \,p_\mu(x,y) + (y^2-x^2)\,\partial_xp_\mu(x,y) \Big)\,dy \,.
    \end{align*}
    Specifically this must hold at $x=0$, giving
    \begin{align*} 
        0 &= -2 \int_{\Omega} y^2\,p_\mu(0,y)\,dy 
        + \int_{\Omega} y^2\,\partial_xp_\mu(0,y) \,dy \\
        \iff 0 &= \int_{\Omega} y^2\,\Big(-2p_\mu(0,y) + \partial_xp_\mu(0,y) \Big) \,dy \,.
    \end{align*}
    Using the identity for the derivative,
    \begin{align*}
        -2p_\mu(0,y) + \partial_xp_\mu(0,y) &= -p_\mu(0,y)\bigg(\frac{\beta(b-c)}{2} (1- p_\mu(0,y)) + 2\bigg)
    \end{align*}
    then since $p_\mu(0,y) \in (0,1)$ for all $y$, and $b > c$, the whole term is negative. Therefore,
    \begin{align*}
        \int_{\Omega} y^2\,\Big(-p_\mu(0,y)\big(\frac{\beta(b-c)}{2} (1- p_\mu(0,y)) + 2\big) \Big) \,dy <0
    \end{align*}
    which is a contradiction.
\end{proof}

\subsubsection*{Linearisation about the Steady State.}
    
Let $\mu^*$ be the solution to the steady state equations. Define $\mu = \mu^* + \varepsilon \eta$. Then,
\begin{align*}
    m_\mu &= \int_{\Omega} z \mu(z) dz\\
    &= \int_{\Omega} z \mu^* + \varepsilon z \eta \;dz\\
    &= m_{\mu^*}  + \varepsilon m_\eta
\end{align*}
and 
\begin{align*}
    p_\mu(x,y) &= \frac{1}{1+\exp(-\beta(y-x)[\int(b-c)z\mu(z)dz-c(y+x)])}\\
    &= \frac{1}{1+\exp(-\beta(y-x)[S_{\mu^*} +\varepsilon\int(b-c)z\eta(z)dz])}.
\end{align*}
The flux can be expressed as
\begin{align*}
    F_\mu(t,x) &= \mu(t,x)V_\mu(t,x)- \frac{\sigma^2}{2}\partial_x\mu.
\end{align*}
The Gateaux derivative in the direction of $\eta$ is 
\begin{align*}
     dF_{\mu^*}[\eta](t,x) &= \eta(t,x) V_{\mu^*}(t,x) +{\mu^*}(t,x)dV_{\mu^*}[\eta](t,x)  - \frac{\sigma^2}{2}\partial_x\eta.
\end{align*}
Applying the quotient rule to the velocity derivative, 
\begin{align*}
    dV_{\mu^*}[\eta](t,x) &= \frac{(x+m_{\mu^*})dI_{\mu^*}[\eta](t,x) - I_{\mu^*}(t,x)m_\eta}{(x+m_{\mu^*})^2} 
\end{align*}
The Gateaux derivative of the integral term $I$ is
\begin{align*}
    dI_{\mu^*}[\eta](t,x) &= \lim_{\varepsilon\rightarrow0}\frac{1}{\varepsilon}\Big(\int_\Omega(y^2-x^2)(\mu^*+\varepsilon\eta)p_{\mu^*+\varepsilon\eta} dy - \int_\Omega(y^2-x^2)\mu^*p_{\mu^*} dy\Big)\\
    &= \lim_{\varepsilon\rightarrow0}\frac{1}{\varepsilon}\Big(\int_\Omega(y^2-x^2)\mu^*(p_{\mu^*+\varepsilon\eta}-p_{\mu^*})dy + \varepsilon \int_\Omega(y^2-x^2)\eta p_{\mu^*+\varepsilon\eta} dy\Big)\\
    &=\beta(b-c)m_\eta\int_\Omega(y^2-x^2)\mu^*(y-x)p_{\mu^*}(1-p_{\mu^*})dy + \int_\Omega(y^2-x^2)\eta p_{\mu^*} dy.
\end{align*}
Therefore the Gateaux derivative of the flux, $dF_{\mu^*}[\eta](t,x)$, is
\begin{align*}
      &= \eta V_{\mu^*} + \frac{\mu^*}{{x+m_{\mu^*}} }\int_\Omega(y^2-x^2)[\beta(b-c)m_\eta(y-x)\mu^*p_{\mu^*}(1-p_{\mu^*}) + \eta p_{\mu^*}]dy\\& - \frac{\mu^*I_{\mu^*} m_\eta}{(x+m_{\mu^*})^2} - \frac{\sigma^2}{2}\partial_x\eta 
\end{align*}
The form of the PDE then follows from plugging the perturbation into the IVP:
\begin{align*}
    \partial_t(\mu^* + \varepsilon \eta) + \partial_x( F_{\mu^* + \varepsilon \eta}) &=0\\
        \partial_t\mu^* + \varepsilon \partial_t\eta + \partial_x(F_{\mu^*}+\varepsilon dF_{\mu^*}(t,x;\eta) + o(\varepsilon)) &=0\\
        \partial_t\eta + \partial_x dF_{\mu^*}(t,x;\eta) &=0
\end{align*}
and defining the no-noise part of $dF_{\mu^*}[\eta](t,x)$ as the operator $\mathcal{L}_{\mu^*}[\eta]$.

\subsubsection*{Proof of Proposition \ref{prop: linear stability, operator L is bounded}}
\begin{proof}
First note that
\begin{align*}
    \bigg\|\frac{\mu^*}{x+m_{\mu^*}}\bigg\|_2 \leq \frac{\|\mu^*\|_2}{m^*}
\end{align*}
We will bound each term in turn. Firstly,
\begin{align*}
 \| \eta V_{\mu^*}\|_2 &= 
    \bigg\|\frac{\eta(x)}{x + m_{\mu^*}}\int_\Omega(y^2-x^2)\mu^*(y)p_{\mu^*}(x,y)\; dy\bigg\|_2\\ &\leq \bigg\|\frac{\eta(x)}{x + m_{\mu^*}}\bigg\|_2\bigg\|\int_\Omega(y^2-x^2)\mu(y)p_{\mu^*}(x,y)\; dy\bigg\|_2\\
    &\leq \bigg\|\frac{\eta(x)}{x + m_{\mu^*}}\bigg\|_2\\
    &\leq \frac{\|\eta\|_2}{m^*}
\end{align*}
The terms within the integral can be separated and bounded as follows: 
\begin{align*}
    \bigg|\int_\Omega(y^2-x^2) \eta(y) p_{\mu^*}\; dy\bigg| &\leq \bigg(\int_\Omega(y^2-x^2)^2  p_{\mu^*}\; dy\bigg)^{1/2}\bigg(\int_\Omega\eta^2(y) dy\bigg)^{1/2}\\
    &\leq \bigg(\int_\Omega(y^2-x^2)^2  dy\bigg)^{1/2}\|\eta\|_2 \\
    &\leq \|\eta\|_2
\end{align*}
and
\begin{align*}
    &\bigg|\beta(b-c)m_\eta\int_\Omega(y^2-x^2)\mu^*(y)p_{\mu^*}(1-p_{\mu^*})(y-x)\; dy\bigg|\\ &\leq \beta(b-c)|m_\eta|\bigg|\int_\Omega(y^2-x^2)\mu^*(y)p_{\mu^*}(1-p_{\mu^*})(y-x)dy\bigg|\\
    &\leq \frac{\beta(b-c)}{2} \bigg|\int_\Omega z\eta(z) dz\bigg|\\
    &\leq \frac{\beta(b-c)}{2}\bigg(\int_\Omega z^2dz\bigg)^{1/2}\|\eta\|_2\\
    &= \frac{\beta (b-c)}{2\sqrt{3}}\|\eta\|_2
\end{align*}
The final term is
\begin{align*}
&\bigg|\frac{\mu^*m_\eta}{(x + m_{\mu^*})^2}\int_\Omega(y^2-x^2)\mu^*(y)p_{\mu^*}(x,y)\; dy\bigg|\\ &\leq \bigg|\frac{\mu^*m_\eta}{(x + m_{\mu^*})^2}\bigg|\bigg|\int_\Omega(y^2-x^2)\mu^*(y)p_{\mu^*}(x,y)\; dy\bigg| \\
&\leq \frac{\|\mu^*\|_2}{(m^*)^2} \|\eta\|_2
\end{align*}
Therefore
\begin{align*}
    \bigg\|\mathcal{L}[\eta](x)\bigg\|_2 &\leq  \frac{\|\eta\|_2}{m^*} +  \frac{\|\mu^*\|_2}{m^*}\Big(\|\eta\|_2 + \frac{\beta (b-c)}{2\sqrt{3}}\|\eta\|_2\Big)  + \frac{\|\mu^*\|_2}{(m^*)^2} \|\eta\|_2\\
    &= \bigg(\frac{1}{m^*} +\frac{\|\mu^*\|_2}{m^*}(1 + \frac{\beta (b-c)}{2\sqrt{3}}) +  \frac{\|\mu^*\|_2}{(m^*)^2}\bigg) \|\eta\|_2 \\
    &=: M_\mathcal{L} \|\eta\|_2
\end{align*}
\end{proof}

\subsubsection*{Proof of Proposition \ref{prop: linear stability, a is bilinear}}
\begin{proof}
 For $\eta,\varphi \in V$,
 \begin{align*}
     |a(\eta,\varphi)| &= \bigg|\frac{\sigma^2}{2}\int_\Omega \eta_x \varphi_x dx - \int_\Omega \mathcal{L}[\eta] \varphi_x dx\bigg|\\
     &\leq \frac{\sigma^2}{2}\bigg|\int_\Omega \eta_x \varphi_x dx\bigg| + \bigg|\int_\Omega \mathcal{L}[\eta] \varphi_x dx\bigg|\\
     &\leq \frac{\sigma^2}{2} \|\eta_x\|_2 \|\varphi_x\|_2 + \|\mathcal{L}[\eta](x)\|_2 \|\varphi\|_2\\
     &\leq\frac{\sigma^2}{2} \|\eta\|_U \|\varphi\|_U+ M_\mathcal{L}\|\eta\|_U\|\varphi\|_U\\
    &= c\;\|\eta\|_U\|\varphi\|_U.
 \end{align*}
 For coerciveness, take $\eta \in U$. Then
 \begin{align*}
     a(\eta,\eta) &= \frac{\sigma^2}{2} \|\eta_x\|_2^2 - \int_\Omega \mathcal{L}[\eta](x) \eta_x dx.
 \end{align*}
 The magnitude of the second term is bounded by
 \begin{align*}
     \bigg|\int_\Omega \mathcal{L}[\eta](x) \eta_x dx\bigg| &\leq \|\mathcal{L}[\eta]\|_2\|v_x\|_2\\
     &\leq M_\mathcal{L}\|\eta\|_2 \|\eta_x\|_2\\
     &\leq \frac{\sigma^2}{4}\|\eta_x\|_2^2 + \frac{M_\mathcal{L}^2}{\sigma^2}\|\eta\|_2^2,
 \end{align*}
 where in the last line we make use of Young's inequality with $\varepsilon = \sigma^2/2$. Therefore,
 \begin{align*}
     a(\eta,\eta) &\geq \frac{\sigma^2}{2} \|\eta_x\|_2^2 - \frac{\sigma^2}{4}\|\eta_x\|_2^2 + \frac{M_\mathcal{L}^2}{\sigma^2}\|\eta\|_2^2\\
     &= \frac{\sigma^2}{4}\|\eta_x\|_2^2 + \frac{M_\mathcal{L}^2}{\sigma^2}\|\eta\|_2^2
 \end{align*}
 Let $\lambda = \frac{M_\mathcal{L}^2}{\sigma^2} + 1$, then
 \begin{align*}
     a(\eta,\eta) + \lambda\|\eta\|_2^2 &\geq \frac{\sigma^2}{4}\|\eta_x\|_2^2 + \|\eta\|_2^2\\
     &\geq \alpha(\|\eta_x\|_2^2 + \|\eta\|_2^2)\\
     &= \alpha \|\eta\|_U^2,
 \end{align*}
 where $\alpha = \min\{\sigma^2/4, 1\}$.
\end{proof}

\subsubsection*{Proof of Theorem \ref{thm: solution to linear pertubation equation}}

\begin{proof}
    The operator $\mathcal{L}: U \rightarrow H$ is bounded by Proposition \ref{prop: linear stability, operator L is bounded}, so the bilinear form 
    \begin{align*}
    a(u,\varphi) = \frac{\sigma^2}{2}\int_\Omega u_x \varphi_x dx - \int_\Omega \mathcal{L}[u](x) \varphi_x(x) dx.
    \end{align*}
    is continuous on $U$. Moreover, for $\sigma^2 >0$, Proposition \ref{prop: linear stability, a is bilinear} gives 
    the necessary continuity and coercivity conditions on $a$. As such, the conditions of Theorem 4.1 in \cite{LionsMagenes1972} are satisfied, hence for every $\eta_0 \in H$ there exists a unique $\eta \in L^2(0,T;U)\; \cap\; H^1(0,T;U')$ solving
    \begin{align*}
        \partial_t\eta + \mathcal{A}\eta = 0,\quad \eta(0) = \eta_0.
    \end{align*}
    We have used the extension to non-zero initial condition given by Remark 4.3 of \cite{LionsMagenes1972}.
\end{proof}

\subsubsection*{Proof of Proposition \ref{prop: decaying perturbation}}

\begin{proof}
 Taking the time derivative of the L2 norm,
\begin{align*}
    \frac{d}{dt} \frac{1}{2}\|\eta\|^2_2 &= \frac{d}{dt}\Big(\int_\Omega \eta^2(x) dx \Big)\\
    &= \int_\Omega \eta(x)\partial_t\eta(x) dx\\
    &= - \int_\Omega \eta(x) \partial_xF[\eta](x) dx\\
    &=  \int_\Omega \partial_x\eta(x) F[\eta](x) dx\\
     &=  \int_\Omega \partial_x\eta(x) (\mathcal{L}[\eta](x)- \frac{\sigma^2}{2}\partial_x \eta(x))\;dx\\
     &= \int_\Omega \partial_x\eta(x)\mathcal{L}[\eta](x)dx - \int_\Omega\frac{\sigma^2}{2}(\partial_x \eta(x))^2 dx\\
     &\leq  M_{\mathcal{L}}\|\eta\|_2\|\partial_x \eta\|_2 - \frac{\sigma^2}{2}\|\partial_x\eta\|_2^2\\
     &\leq  \|\partial_x \eta\|_2 \Big(M_{\mathcal{L}}\|\eta\|_2- \frac{\sigma^2}{2}\|\partial_x\eta\|_2\Big).\\
     \intertext{Since the total mass of the perturbation is zero, by the Poincare–Wirtinger inequality \cite{evans2022partial} there exists a constant $C$ such that}
     &\leq \|\partial_x \eta\|^2_2 \Big(M_{\mathcal{L}}C- \frac{\sigma^2}{2}\Big).
\end{align*}
The coefficient is negative for $\sigma^2 > 2CM_{\mathcal{L}}$, so there exists a $\gamma >0$ such that 
\begin{align*}
    \frac{d}{dt} \frac{1}{2}\|\eta\|^2_2 &\leq -\gamma\|\partial_x \eta\|^2_2 \leq -\frac{\gamma}{C^2}\|\eta\|_2^2,
\end{align*}
and by Gr\"{o}nwall's inequality,
\begin{align*}
    \|\eta(t)\|^2_2 \leq e^{-\frac{\gamma}{C^2}t} \|\eta_0\|_2^2.
\end{align*}
\end{proof}

\end{document}

%% file: preamble.tex
\usepackage[utf8]{inputenc}
\usepackage[dvipsnames]{xcolor}
\usepackage{amsfonts}
\usepackage{amsthm}
\usepackage{amsmath}
\usepackage{parskip}
\usepackage{mathrsfs}
\usepackage{amssymb}
\usepackage{xtab}
\usepackage{setspace}
\usepackage{datetime}
\usepackage{svg}
\usepackage{csquotes}
\usepackage{fancyhdr}
\usepackage{float}
\usepackage{comment}
\usepackage{enumitem}
\usepackage{mathalfa}
\usepackage{array}
\usepackage{titlesec}
\usepackage{mdframed}
\usepackage{listings}
\usepackage{booktabs}
\usepackage{dsfont}
\usepackage{textcomp}
\usepackage{graphicx}
\usepackage{wrapfig}
\usepackage{cases}
\usepackage{bm}

\usepackage{hyperref}
\usepackage{cleveref}

\usepackage[font=small,labelfont=bf]{caption}

\usepackage[utf8]{inputenc}
\usepackage[english]{babel}

\usepackage{placeins}
\usepackage{subcaption}
\usepackage{multirow}

\graphicspath{{Figures/}}

\allowdisplaybreaks

\usepackage{thmtools}
\usepackage{thm-restate}
\declaretheorem[name=Proposition,numberwithin=section]{proposition}
\declaretheorem[name=Lemma,numberwithin=section]{lemma}
\declaretheorem[name=Theorem,numberwithin=section]{theorem}

\declaretheorem[name=Corollary,numberwithin=section]{corollary}
\declaretheorem[name=Assumption]{assumption}

\newcounter{subassumption}

\usepackage{circuitikz}
\usetikzlibrary{decorations.pathreplacing,calligraphy}

\newcommand\acksname{Acknowledgments}
\specialcomment{acks}{%
  \begingroup
  \section*{\acksname}
  \phantomsection\addcontentsline{toc}{section}{\acksname}
}{%
  \endgroup
}

%% file: bibliography.bib
@book{evans2022partial,
  title={Partial differential equations},
  author={Evans, Lawrence C},
  volume={19},
  year={2022},
  publisher={American Mathematical Society}
}

@article{Bonnet2017ThePM,
  title={The Pontryagin Maximum Principle in the Wasserstein Space},
  author={Beno{\^i}t Bonnet and Francesco Rossi},
  journal={Calculus of Variations and Partial Differential Equations},
  year={2017},
  volume={58},
  pages={1-36},
  url={https://api.semanticscholar.org/CorpusID:23997334}
}

@book{LionsMagenes1972,
  author = {Lions, J. L. and Magenes, E.},
  title = {Non-Homogeneous Boundary Value Problems and Applications},
  publisher = {Springer},
  year = {1972},
  volume = {1}
}

@article{bara_enabling_2022,
    title = {Enabling imitation-based cooperation in dynamic social networks},
    volume = {36},
    issn = {1573-7454},
    url = {https://doi.org/10.1007/s10458-022-09562-w},
    doi = {10.1007/s10458-022-09562-w},
    language = {en},
    number = {2},
    journal = {Autonomous Agents and Multi-Agent Systems},
    author = {Bara, Jacques and Turrini, Paolo and Andrighetto, Giulia},
    month = may,
    year = {2022},
    pages = {34},
}

@article{leung_learning_2024,
    title = {Learning {Partner} {Selection} {Rules} that {Sustain} {Cooperation}  in {Social} {Dilemmas} with the {Option} of {Opting} {Out}},
    language = {en},
    journal = {New Zealand},
    author = {Leung, Chin-wing},
    year = {2024},
}

@article{zheng_simple_2017,
    title = {A simple rule of direct reciprocity leads to the stable coexistence of cooperation and defection in the {Prisoner}'s {Dilemma} game},
    volume = {420},
    issn = {00225193},
    url = {https://linkinghub.elsevier.com/retrieve/pii/S0022519317300991},
    doi = {10.1016/j.jtbi.2017.02.036},
    language = {en},
    journal = {Journal of Theoretical Biology},
    author = {Zheng, Xiu-Deng and Li, Cong and Yu, Jie-Ru and Wang, Shi-Chang and Fan, Song-Jia and Zhang, Bo-Yu and Tao, Yi},
    month = may,
    year = {2017},
    pages = {12--17},
}

@article{rand_dynamic_2011,
    title = {Dynamic social networks promote cooperation in experiments with humans},
    volume = {108},
    issn = {0027-8424, 1091-6490},
    url = {https://pnas.org/doi/full/10.1073/pnas.1108243108},
    doi = {10.1073/pnas.1108243108},
    language = {en},
    number = {48},
    journal = {Proceedings of the National Academy of Sciences},
    author = {Rand, David G. and Arbesman, Samuel and Christakis, Nicholas A.},
    month = nov,
    year = {2011},
    pages = {19193--19198},
}

@article{santos_cooperation_2006,
    title = {Cooperation {Prevails} {When} {Individuals} {Adjust} {Their} {Social} {Ties}},
    volume = {2},
    issn = {1553-7358},
    url = {https://dx.plos.org/10.1371/journal.pcbi.0020140},
    doi = {10.1371/journal.pcbi.0020140},
    language = {en},
    number = {10},
    journal = {PLoS Computational Biology},
    author = {Santos, Francisco C and Pacheco, Jorge M and Lenaerts, Tom},
    editor = {Amaral, Luis},
    month = oct,
    year = {2006},
    pages = {e140},
}

@misc{anastassacos_partner_2020,
    type = {Proceedings paper},
    title = {Partner {Selection} for the {Emergence} of {Cooperation} in {Multi}-{Agent} {Systems} {Using} {Reinforcement} {Learning}},
    copyright = {open},
    url = {https://aaai.org/Library/conferences-library.php},
    language = {eng},
    journal = {Proceedings of the Thirty-Fourth AAAI Conference on Artificial Intelligence (AAAI-20)},
    publisher = {Advancement of Artificial Intelligence (AAAI)},
    author = {Anastassacos, N. and Hailes, S. and Musolesi, M.},
    month = feb,
    year = {2020},
    note = {Conference Name: Thirty-Fourth AAAI Conference on Artificial Intelligence (AAAI-20)
Meeting Name: Thirty-Fourth AAAI Conference on Artificial Intelligence (AAAI-20)
Place: New York City, NY, USA},
}

@article{nowak_five_2006,
    title = {Five rules for the evolution of cooperation},
    volume = {314},
    issn = {0036-8075},
    url = {https://pmc.ncbi.nlm.nih.gov/articles/PMC3279745/},
    doi = {10.1126/science.1133755},
    number = {5805},
    journal = {Science (New York, N.y.)},
    author = {Nowak, Martin A.},
    month = dec,
    year = {2006},
    pages = {1560--1563},
}

@article{zhang_opting_2016,
    title = {Opting out against defection leads to stable coexistence with cooperation},
    volume = {6},
    issn = {2045-2322},
    url = {https://www.ncbi.nlm.nih.gov/pmc/articles/PMC5075917/},
    doi = {10.1038/srep35902},
    journal = {Scientific Reports},
    author = {Zhang, Bo-Yu and Fan, Song-Jia and Li, Cong and Zheng, Xiu-Deng and Bao, Jian-Zhang and Cressman, Ross and Tao, Yi},
    month = oct,
    year = {2016},
    pages = {35902},
}

@article{graser_repeated_2025,
    title = {Repeated games with partner choice},
    volume = {21},
    issn = {1553-7358},
    url = {https://journals.plos.org/ploscompbiol/article?id=10.1371/journal.pcbi.1012810},
    doi = {10.1371/journal.pcbi.1012810},
    language = {en},
    number = {2},
    journal = {PLOS Computational Biology},
    publisher = {Public Library of Science},
    author = {Graser, Christopher and Fujiwara-Greve, Takako and García, Julián and Veelen, Matthijs van},
    month = feb,
    year = {2025},
    pages = {e1012810},
}

@article{segbroeck_coevolution_2010,
    title = {Coevolution of {Cooperation}, {Response} to {Adverse} {Social} {Ties} and {Network} {Structure}},
    volume = {1},
    copyright = {http://creativecommons.org/licenses/by/3.0/},
    issn = {2073-4336},
    url = {https://www.mdpi.com/2073-4336/1/3/317},
    doi = {10.3390/g1030317},
    language = {en},
    number = {3},
    journal = {Games},
    publisher = {Molecular Diversity Preservation International},
    author = {Segbroeck, Sven Van and Santos, Francisco C. and Pacheco, Jorge M. and Lenaerts, Tom},
    month = sep,
    year = {2010},
    pages = {317--337},
}

@article{nugent_opinion_2025,
  title={Opinion dynamics with continuous age structure},
  author={Nugent, Andrew and Gomes, Susana N and Wolfram, Marie-Therese},
  journal={European Journal of Applied Mathematics},
  pages={1--36},
  publisher={Cambridge University Press},
  year={2025}
}

@inproceedings{leung_modelling_2022,
    address = {Vienna, Austria},
    title = {Modelling the {Dynamics} of {Multi}-{Agent} {Q}-learning: {The} {Stochastic} {Effects} of {Local} {Interaction} and {Incomplete} {Information}},
    isbn = {978-1-956792-00-3},
    shorttitle = {Modelling the {Dynamics} of {Multi}-{Agent} {Q}-learning},
    url = {https://www.ijcai.org/proceedings/2022/55},
    doi = {10.24963/ijcai.2022/55},
    language = {en},
    booktitle = {Proceedings of the {Thirty}-{First} {International} {Joint} {Conference} on {Artificial} {Intelligence}},
    publisher = {International Joint Conferences on Artificial Intelligence Organization},
    author = {Leung, Chin-wing and Hu, Shuyue and Leung, Ho-fung},
    month = jul,
    year = {2022},
    pages = {384--390},
}

@inproceedings{hu_modelling_2019,
 author = {Hu, Shuyue and Leung, Chin-wing and Leung, Ho-fung},
 booktitle = {Advances in Neural Information Processing Systems},
 editor = {H. Wallach and H. Larochelle and A. Beygelzimer and F. d\textquotesingle Alch\'{e}-Buc and E. Fox and R. Garnett},
 pages = {},
 publisher = {Curran Associates, Inc.},
 title = {Modelling the Dynamics of Multiagent Q-Learning in Repeated Symmetric Games: a Mean Field Theoretic Approach},
 url = {https://proceedings.neurips.cc/paper_files/paper/2019/file/40afd3a37cca05efe623b7509855c73a-Paper.pdf},
 volume = {32},
 year = {2019}
}

@book{pavliotis_stuart_2008,
place={New York},
title={Multiscale Methods: Averaging and Homogenization}, ISBN={978-0-387-73829-1},
DOI={10.1007/978-0-387-73829-1},
publisher={Springer},
author={Pavliotis, Grigorios A. and Stuart, Andrew M.}, year={2008}
}

@article{nugent2024bridging,
  title={Bridging the gap between agent based models and continuous opinion dynamics},
  author={Nugent, Andrew and Gomes, Susana N and Wolfram, Marie-Therese},
  journal={Physica A: Statistical Mechanics and its Applications},
  volume={651},
  pages={129886},
  year={2024},
  publisher={Elsevier}
}

@article{santos_scale-free_2005,
    title = {Scale-{Free} {Networks} {Provide} a {Unifying} {Framework} for the {Emergence} of {Cooperation}},
    volume = {95},
    copyright = {http://link.aps.org/licenses/aps-default-license},
    issn = {0031-9007, 1079-7114},
    url = {https://link.aps.org/doi/10.1103/PhysRevLett.95.098104},
    doi = {10.1103/PhysRevLett.95.098104},
    language = {en},
    number = {9},
    journal = {Physical Review Letters},
    author = {Santos, F. C. and Pacheco, J. M.},
    month = aug,
    year = {2005},
    pages = {098104},
}

@article{fotouhi_evolution_2019,
    title = {Evolution of cooperation on large networks with community structure},
    volume = {16},
    issn = {1742-5689},
    url = {https://doi.org/10.1098/rsif.2018.0677},
    doi = {10.1098/rsif.2018.0677},
    number = {152},
    journal = {Journal of The Royal Society Interface},
    author = {Fotouhi, Babak and Momeni, Naghmeh and Allen, Benjamin and Nowak, Martin A.},
    month = mar,
    year = {2019},
    pages = {20180677},
}

@article{ohtsuki_simple_2006,
    title = {A simple rule for the evolution of cooperation on graphs and social networks},
    volume = {441},
    copyright = {2006 Springer Nature Limited},
    issn = {1476-4687},
    url = {https://www.nature.com/articles/nature04605},
    doi = {10.1038/nature04605},
    language = {en},
    number = {7092},
    journal = {Nature},
    publisher = {Nature Publishing Group},
    author = {Ohtsuki, Hisashi and Hauert, Christoph and Lieberman, Erez and Nowak, Martin A.},
    month = may,
    year = {2006},
    pages = {502--505},
}

@article{melamed_prosocial_2017,
    title = {Prosocial {Orientation} {Alters} {Network} {Dynamics} and {Fosters} {Cooperation}},
    volume = {7},
    copyright = {2017 The Author(s)},
    issn = {2045-2322},
    url = {https://www.nature.com/articles/s41598-017-00265-x},
    doi = {10.1038/s41598-017-00265-x},
    language = {en},
    number = {1},
    journal = {Scientific Reports},
    publisher = {Nature Publishing Group},
    author = {Melamed, David and Simpson, Brent and Harrell, Ashley},
    month = mar,
    year = {2017},
    pages = {357},
}

@inproceedings{defection_russell2026,
         journal = {Proceedings of the 25th International Conference on Autonomous Agents and Multiagent Systems},
            note = {In Press},
       booktitle = {25th International Conference on Autonomous Agents and Multiagent Systems},
           month = {05},
            year = {2026},
           title = {Defection at first sight : learning partner selection in optional social dilemmas without prior information},
       publisher = {IFAAMAS; ACM Digital library},
             doi = {10.65109/IBSZ1473},
             url = {https://doi.org/10.65109/IBSZ1473},
          author = {Russell, Benedict and Leung, Chin-wing and Turrini, Paolo}
}

@article{nugent_evolving_2023,
    title = {On evolving network models and their influence on opinion formation},
    volume = {456},
    issn = {0167-2789},
    url = {https://www.sciencedirect.com/science/article/pii/S0167278923002683},
    doi = {10.1016/j.physd.2023.133914},
    journal = {Physica D: Nonlinear Phenomena},
    author = {Nugent, Andrew and Gomes, Susana N. and Wolfram, Marie-Therese},
    month = dec,
    year = {2023},
    pages = {133914},
}

@article{li_evolution_2020,
    title = {Evolution of cooperation on temporal networks},
    volume = {11},
    copyright = {2020 The Author(s)},
    issn = {2041-1723},
    url = {https://www.nature.com/articles/s41467-020-16088-w},
    doi = {10.1038/s41467-020-16088-w},
    language = {en},
    number = {1},
    journal = {Nature Communications},
    publisher = {Nature Publishing Group},
    author = {Li, Aming and Zhou, Lei and Su, Qi and Cornelius, Sean P. and Liu, Yang-Yu and Wang, Long and Levin, Simon A.},
    month = may,
    year = {2020},
    pages = {2259},
}

@article{pacheco_active_2006,
    title = {Active linking in evolutionary games},
    volume = {243},
    issn = {0022-5193},
    url = {https://www.sciencedirect.com/science/article/pii/S0022519306002736},
    doi = {10.1016/j.jtbi.2006.06.027},
    number = {3},
    journal = {Journal of Theoretical Biology},
    author = {Pacheco, Jorge M. and Traulsen, Arne and Nowak, Martin A.},
    month = dec,
    year = {2006},
    pages = {437--443},
}

@article{pinheiro_linking_2016,
    title = {Linking {Individual} and {Collective} {Behavior} in {Adaptive} {Social} {Networks}},
    volume = {116},
    copyright = {http://link.aps.org/licenses/aps-default-license},
    issn = {0031-9007, 1079-7114},
    url = {https://link.aps.org/doi/10.1103/PhysRevLett.116.128702},
    doi = {10.1103/PhysRevLett.116.128702},
    language = {en},
    number = {12},
    journal = {Physical Review Letters},
    author = {Pinheiro, Flávio L. and Santos, Francisco C. and Pacheco, Jorge M.},
    month = mar,
    year = {2016},
    pages = {128702},
}

@article{izquierdo_leave_2014,
    title = {Leave and let leave: {A} sufficient condition to explain the evolutionary emergence of cooperation},
    volume = {46},
    issn = {0165-1889},
    shorttitle = {Leave and let leave},
    url = {https://www.sciencedirect.com/science/article/pii/S0165188914001456},
    doi = {10.1016/j.jedc.2014.06.007},
    journal = {Journal of Economic Dynamics and Control},
    author = {Izquierdo, Luis R. and Izquierdo, Segismundo S. and Vega-Redondo, Fernando},
    month = sep,
    year = {2014},
    pages = {91--113},
}

@inproceedings{fan_colearning_2025,
month = {August},
publisher = {International Joint Conferences on Artificial Intelligence Organization},
doi = {10.24963/ijcai.2025/9},
pages = {72--80},
booktitle = {34th International Joint Conference on Artificial Intelligence},
note = {In Press},
title = {Co-learning of strategy and structure achieves full cooperation in complex networks with dynamical linking},
year = {2025},
journal = {Proceedings of the Thirty-Fourth International Joint Conference on Artificial Intelligence},
url = {https://doi.org/10.24963/ijcai.2025/9},
author = {Fan, Xiaoqing and Leung, Chin-wing and Turrini, Paolo}
}

@article{fehl_co-evolution_2011,
    title = {Co-evolution of behaviour and social network structure promotes human cooperation},
    volume = {14},
    issn = {1461-0248},
    doi = {10.1111/j.1461-0248.2011.01615.x},
    language = {eng},
    number = {6},
    journal = {Ecology Letters},
    author = {Fehl, Katrin and van der Post, Daniel J. and Semmann, Dirk},
    month = jun,
    year = {2011},
    pages = {546--551},
}

@article{kabir_how_2021,
    title = {How evolutionary game could solve the human vaccine dilemma},
    volume = {152},
    issn = {0960-0779},
    url = {https://www.sciencedirect.com/science/article/pii/S0960077921008134},
    doi = {10.1016/j.chaos.2021.111459},
    journal = {Chaos, Solitons \& Fractals},
    author = {Kabir, K. M. Ariful},
    month = nov,
    year = {2021},
    pages = {111459},
}

@article{gottlieb_tax_1985,
    title = {Tax evasion and the prisoner's dilemma},
    volume = {10},
    issn = {0165-4896},
    url = {https://www.sciencedirect.com/science/article/pii/0165489685900393},
    doi = {10.1016/0165-4896(85)90039-3},
    number = {1},
    journal = {Mathematical Social Sciences},
    author = {Gottlieb, Daniel},
    month = aug,
    year = {1985},
    pages = {81--89},
}

@article{conybeare_public_1984,
    title = {Public {Goods}, {Prisoners}' {Dilemmas} and the {International} {Political} {Economy}},
    volume = {28},
    issn = {0020-8833},
    url = {https://doi.org/10.2307/2600395},
    doi = {10.2307/2600395},
    number = {1},
    journal = {International Studies Quarterly},
    author = {Conybeare, John A. C.},
    month = mar,
    year = {1984},
    pages = {5--22},
}

@incollection{jabin_meanfieldlimit_2017,
    address = {Cham},
    title = {Mean {Field} {Limit} for {Stochastic} {Particle} {Systems}},
    isbn = {978-3-319-49994-9 978-3-319-49996-3},
    url = {http://link.springer.com/10.1007/978-3-319-49996-3_10},
    doi = {10.1007/978-3-319-49996-3_10},
    language = {en},
    booktitle = {Active {Particles}, {Volume} 1},
    publisher = {Springer International Publishing},
    author = {Jabin, Pierre-Emmanuel and Wang, Zhenfu},
    year = {2017},
    note = {Series Title: Modeling and Simulation in Science, Engineering and Technology},
    pages = {379--402},
}

@Inbook{Villani2009,
author="Villani, C{\'e}dric",
title="The Wasserstein distances",
bookTitle="Optimal Transport: Old and New",
year="2009",
publisher="Springer Berlin Heidelberg",
address="Berlin, Heidelberg",
pages="93--111",
isbn="978-3-540-71050-9",
doi="10.1007/978-3-540-71050-9_6",
url="https://doi.org/10.1007/978-3-540-71050-9_6"
}

@misc{russell2026dynamicspolicygradient,
      title={The Dynamics of Policy Gradient in Social Dilemmas with Partner Selection}, 
      author={Benedict Russell and Chin-wing Leung and Paolo Turrini},
      year={2026},
      eprint={2605.18185},
      archivePrefix={arXiv},
      url={https://arxiv.org/abs/2605.18185}, 
}

@article{guo_network_2023,
    title = {Network adaption based on environment feedback promotes cooperation in co-evolutionary games},
    volume = {617},
    issn = {0378-4371},
    url = {https://www.sciencedirect.com/science/article/pii/S0378437123002443},
    doi = {10.1016/j.physa.2023.128689},
    journal = {Physica A: Statistical Mechanics and its Applications},
    author = {Guo, Yujie and Zhang, Liming and Li, Haihong and Dai, Qionglin and Yang, Junzhong},
    month = may,
    year = {2023},
    pages = {128689},
}

@article{rand_static_2014,
    title = {Static network structure can stabilize human cooperation},
    volume = {111},
    url = {https://www.pnas.org/doi/10.1073/pnas.1400406111},
    doi = {10.1073/pnas.1400406111},
    number = {48},
    journal = {Proceedings of the National Academy of Sciences},
    publisher = {Proceedings of the National Academy of Sciences},
    author = {Rand, David G. and Nowak, Martin A. and Fowler, James H. and Christakis, Nicholas A.},
    month = dec,
    year = {2014},
    pages = {17093--17098},
}

@misc{nugent2026emergent,
      title={Emergent structures in coupled opinion and network dynamics}, 
      author={Andrew Nugent and Carmen Calatayud Fernandez and Susana N. Gomes},
      year={2026},
      eprint={2602.03738},
      archivePrefix={arXiv},
      url={https://arxiv.org/abs/2602.03738}, 
}

@article{gkogkas2021continuum,
  title={Continuum limits for adaptive network dynamics},
  author={Gkogkas, Marios Antonios and Kuehn, Christian and Xu, Chuang},
  journal={arXiv preprint arXiv:2109.05898},
  year={2021}
}

@article{ayi2021mean,
  title={Mean-field and graph limits for collective dynamics models with time-varying weights},
  author={Ayi, Nathalie and Duteil, Nastassia Pouradier},
  journal={Journal of Differential Equations},
  volume={299},
  pages={65--110},
  year={2021},
  publisher={Elsevier}
}

@article{motsch2014heterophilious,
title={Heterophilious dynamics enhances consensus},
author={Motsch, Sebastien and Tadmor, Eitan},
journal={SIAM review},
volume={56}, 
number={4},
pages={577--621},
year={2014},
publisher={SIAM}
}

@article{brooks2024emergence,
title={Emergence of polarization in a sigmoidal bounded-confidence model of opinion dynamics},
author={Brooks, Heather Z and Chodrow, Philip S and Porter, Mason A},
journal={SIAM Journal on Applied Dynamical Systems},
volume={23},
number={2},
pages={1442--1470},year={2024},
publisher={SIAM}
}

@article{couzin2002collective,
title={Collective memory and spatial sorting in animal groups},
author={Couzin, Iain D and Krause, Jens and James, Richard and Ruxton, Graeme D and Franks, Nigel R},
journal={Journal of theoretical biology},
volume={218}, number={1}, 
pages={1--11}, year={2002},
publisher={Elsevier}
}

@article{HASKOVEC201342,
title = {Flocking dynamics and mean-field limit in the Cucker–Smale-type model with topological interactions},
journal = {Physica D: Nonlinear Phenomena},
volume = {261},
pages = {42-51},
year = {2013},
issn = {0167-2789},
doi = {https://doi.org/10.1016/j.physd.2013.06.006},
url = {https://www.sciencedirect.com/science/article/pii/S0167278913001796},
author = {Jan Haskovec},
}

@article{propogation_review,
title = {Propagation of chaos: A review of models, methods and applications. I. Models and methods},
journal = {Kinetic and Related Models},
volume = {15},
number = {6},
pages = {895-1015},
year = {2022},
issn = {1937-5093},
doi = {10.3934/krm.2022017},
url = {https://www.aimsciences.org/article/id/631fd3b64cedfd0007ce7600},
author = {Louis-Pierre Chaintron and Antoine Diez},
}

@book{shapiro_fixed_point,
author = {Shapiro, Joel},
year = {2016},
month = {01},
pages = {},
publisher = {Springer},
title = {A Fixed-Point Farrago},
isbn = {978-3-319-27976-3},
doi = {10.1007/978-3-319-27978-7}
}

@book{carmona_meanfieldgames,
author = {Carmona, Rene and Delarue, François},
year = {2018},
month = {01},
pages = {},
title = {Probabilistic Theory of Mean Field Games with Applications I : Mean Field FBSDEs, Control, and Games},
publisher = {Springer},
doi = {doi.org/10.1007/978-3-319-58920-6},
}

@misc{izquierdo_successful_2026,
    title = {Successful strategies in the voluntarily repeated {Prisoner}’s {Dilemma}},
    language = {en},
    author = {Izquierdo, Luis R and Izquierdo, Segismundo S and Boyd, Robert},
    year = {2026},
    doi = {10.64898/2026.01.16.699891}
}

@book{rudin_principles,
  author    = {Rudin, Walter},
  title     = {Principles of Mathematical Analysis},
  edition   = {3},
  publisher = {McGraw-Hill},
  year      = {1976},
  isbn      = {978-0070856134}
}

@incollection{meleard_asymptotic_1996,
    address = {Berlin, Heidelberg},
    title = {Asymptotic behaviour of some interacting particle systems; {McKean}-{Vlasov} and {Boltzmann} models},
    isbn = {978-3-540-68513-5},
    url = {https://doi.org/10.1007/BFb0093177},
    doi = {10.1007/BFb0093177},
    language = {en},
    urldate = {2026-06-09},
    booktitle = {Probabilistic {Models} for {Nonlinear} {Partial} {Differential} {Equations}: {Lectures} given at the 1st {Session} of the {Centro} {Internazionale} {Matematico} {Estivo} ({C}.{I}.{M}.{E}.) held in {Montecatini} {Terme}, {Italy}, {May} 22–30, 1995},
    publisher = {Springer},
    author = {Méléard, Sylvie},
    editor = {Graham, Carl and Kurtz, Thomas G. and Méléard, Sylvie and Protter, Philip E. and Pulvirenti, Mario and Talay, Denis and Talay, Denis and Tubaro, Luciano},
    year = {1996},
    pages = {42--95},
}
